\newif\ifpictures
\picturestrue

\documentclass[11pt]{amsart}
\usepackage[T1]{fontenc}
\usepackage{charter}
\usepackage[expert]{mathdesign}
\usepackage{latex_base}
\usepackage{tikz-cd}    % For making diagrams
\usepackage{mathrsfs}
\usepackage{hyperref}
\usepackage{mathtools}
\usepackage{macros}
\usepackage{afterpage}
\usepackage{lipsum} % For side brackets
\usepackage{xcolor}
\usepackage{amsmath}
\usepackage[dash,dot]{dashundergaps}
\usetikzlibrary{patterns}
\usepackage{tikz,ifthen,fp,calc}
\usetikzlibrary{shapes,arrows,shadows,positioning}
\usetikzlibrary{backgrounds}
\title[Describing the Jelonek set of polynomial maps via Newton polytopes]{Describing the Jelonek set of polynomial maps via Newton polytopes.}
\author{Boulos El Hilany}
\thanks{For this work, the author was supported by the Institute of Mathematics, Polish Academy of Sciences}

\begin{document}

\maketitle

\begin{abstract} 
Let $\K=\C$, or $\R$, and $S_f$ be the set of points in $\K^n$ at which a polynomial map $f:\K^n\rightarrow\K^n$ is non-proper. Jelonek proved that $S_f$ is a semi-algebraic set that is ruled by polynomial curves, with $\dim S_f\leq n-1$, and provided a method to compute $S_f$ for $\K = \C$. However, such methods do not exist for $\K = \R$. 

In this paper, we establish a straightforward description of $S_f$ for a large family of non-proper maps $f$ using the Newton polytopes of the polynomials appearing in $f$. Thus resulting in a new method for computing $S_f$ that works for $\K=\R$, and highlights an interplay between the geometry of polytopes and that of $S_f$. As an application, we recover some of Jelonek's results, and provide conditions on (non-)properness of $f$. Moreover, we discover another large family of maps $f$ whose $S_f$ has dimension $n-1$ (even for $\K=\R$), satisfies an explicit stratification, and weak smoothness properties. This novel description allows our tools to be extended to all non-proper maps.
\end{abstract}

 \markleft{}
 \markright{}
\section{Introduction} Let $\K$ denote the field $\C$, or $\R$, and $n$ be a positive integer. A polynomial map $f:\K^n \rightarrow \K^n$, $x=(x_1,\ldots,x_n) \mapsto \big( f_1(x),\ldots, f_n(x)\big)$ is \emph{non-degenerate} if the Jacobian matrix $J(f)=\partial f_i/\partial x_j$ has full rank at generic points $x\in\K^n$. It is said to be \emph{non-proper at} $y\in\K^n$ if for any open neighborhood $\mathcal{U}_y$ of $y$, the preimage $f^{-1}(\cl(\mathcal{U}_y))$ (where $\cl ( . )$ denotes the closure) is non-compact. The Jelonek set $S_f\subset \K^n$ is the set of all points $y$ at which $f$ is non-proper (see Example~\ref{ex:bivar-non-proper}). Our goal is to better understand the geometry of the Jelonek set of non-degenerate polynomial maps $f$, and to introduce a simple, more accessible method for computing it, mainly when $\K = \R$. 

We focus on two families of non-degenerate maps called \emph{$T$-BG}, and \emph{very $T$-BG} (the second is included in the first), each of which forms an open dense subset in the space of non-proper polynomial maps. We will shortly make more precise both, the definition of the above two families, and the following main result.

\begin{theorem}\label{th:main}
Let $f:\K^n\rightarrow\K^n$ be a $T$-BG map. Then, the equations of $S_f\subset\K^n$ can be computed using only the data of $f$ at some faces of the Newton polytopes in $\R^n$ of the polynomials $f_1,\ldots,f_n$. Moreover, we make this computation explicit. 

\end{theorem} This is the only known method to compute $S_f$ for $\K = \R$, and arbitrary $n$, and it relies mainly on tools from convex geometry. Thus providing a useful correspondence between the combinatorics of polytopes in $\R^n$, and the geometry of $S_f$ in $\K^n$. This makes it simpler, and accessible to implement using linear programming tools. As an application, we provide a first step towards classifying non-proper maps whose Jelonek set has prescribed topological properties.
\begin{theorem}\label{th:very-T-BG}
Let $f:\K^n\rightarrow\K^n$ be a very $T$-BG map. Then, the Jelonek set in $\K^n$ has dimension $n-1$ (even for $\K = \R$), admits a stratification that can be described using the geometry of polytopes in $\R^n$, and its singular locus is either empty, or coincides with complete self-intersections. 
\end{theorem} 

\subsection{Motivation and previously-known results} Introduction and study of $S_f$ was originally done by Jelonek~\cite{Jel93,Jel99}, and was motivated by the further understanding of the \emph{Jacobian conjecture}. This notorious problem states that for $\K = \C$, the condition $\det J(f)\in\C^*$ is sufficient for a dominant polynomial map $f$ to be invertible. Attempts to proving it lead to important results, however, only partial ones are known (see e.g.~\cite{BCW82,dEss12}). Although the Jacobian conjecture remains unsolved for $n\geq 2$, the Jelonek set nevertheless constitutes an important invariant for polynomial maps, and it has found its way into several applications in affine geometry. The utility of its computation appears, e.g., in the study of biforcation sets for maps $\K^n\rightarrow\K$~\cite{JelKur14,JelTib17}, tackling the Russell Conjecture about characterization of affine spaces~\cite{Jel10}, and answering several questions about the topology of polynomial maps~\cite{Jel99}.

 It was first shown for the complex case in~\cite{Jel93} that $S_f$ is either empty or it is a \emph{$\C$-uniruled} algebraic hypersurface (i.e. ruled by affine curves $\Gamma$ that are images of polynomial maps $\C\rightarrow\Gamma$). This property turns out to be true also for polynomial maps between affine varieties with arbitrary dimensions over fields in arbitrary characteristics~\cite[Theorem 5.7]{Jel99},~\cite{Sta05,Jel10,JelLas18}. Moreover, the Jelonek set is an $\R$-uniruled semi-algebraic variety for $\K=\R$ that satisfies $1\leq \dim S_f\leq n-1$ \cite{Jel02}. In his original paper on the subject, Jelonek provided an upper bound on the degree of $S_f$, and a method of computation using Gr\"obner bases. Later, Stasica~\cite{Sta02} proved that this method can be used to compute $S_f$ for polynomial maps between affine varieties over an algebraically closed field.

Regarding simpler methods of computation, Jelonek shows in~\cite{Jel01c} that for $n=2$, one can use the resultant of polynomials. For $n\geq 3$, however, Gr\"obner bases remains the only tool employed to compute $S_f$ when $\K=\C$, or any other algebraically closed field. On the other hand, one cannot use it anymore whenever $\K = \R$. Hence, describing the topology of $f$ for the real case, remains a difficult problem in general. This emphasizes the importance of the goal of this paper, especially that polynomial maps between two real spaces appear in applications such as algebraic statistics~\cite{DrStSu08}, and optimization problems~\cite{Las10}.

\subsection{Stratagem} Write $x^a$ for the $n$-variate monomial $x_1^{a_1}\cdots x_n^{a_n}$, with $x=(x_1,\ldots,x_n)$ a coordinate, and $a=(a_1,\ldots,a_n)\in\N^n$. Thus, a polynomial is a finite linear combination \[P(x) = \sum_{a\in\N^n} c_ax^a,\] whose \emph{support}, $\supp P\subset\N^n$, is $\left\lbrace a\in\N^n~|~c_a\neq 0\right\rbrace$, and whose \emph{Newton polytope} $\New P\subset\R^n$ is the convex hull of $ \supp P$. The \emph{Newton tuple} $\underline{\New} f$ of $f$ will refer to the tuple $(\New f_1,\ldots,\New f_n )$.
% We denote by $\underline{\New} f$ the \emph{Newton tuple} Going back to  above, this notation refers to the tuple $\big(\New f_1,\ldots,\New f_n \big)$. 
A \emph{face} of a convex polytope $\Delta\subset\R^n$ is the intersection of $\Delta$ with a hyperplane minimizing the value of some linear function on $\Delta$.

A point in the $n$-torus $T_n=(\K^*)^n$ escapes ``to infinity'' whenever its image reaches a point in $S_f$. This asymptotic behaviour depends on a vector $\alpha\in\R^n$. On the one hand, this vector determines a choice $\gamma_1,\ldots,\gamma_n\subset\R^n$, denoted by $\gamma$, of (not necessarily proper) faces of $\New f_1,\ldots,\New f_n$ respectively, which we call a \emph{tuple of $\underline{\New} f$} (see Section~\ref{subsec:types-faces}). On the other hand, tracking down this escaping point in $T_n$ involves solving the polynomial system 
\begin{equation}\label{eq:sys:0}
\left\{
  \begin{array}{ccc}
   f_1 - y_1 & = & 0,\\
\ & \vdots & \\
  f_n -y_n & = & 0, \\
  \end{array}\right. 
\end{equation} 

expressed as $f-y=\underline{0}$, for some $y\in\K^n$ (treated here as a vector of coefficients). The relation between the two phenomena, detailed in Proposition~\ref{prop:solutions-infty1}, is the following well-known fact from~\cite{Ber75}: Detecting solutions at infinity requires solving, in $T_n$, a \emph{restricted sub-system} $(f-y)_\gamma = \underline{0}$ of~\eqref{eq:sys:0}. This consists of sub-polynomials $f_i -y_i$ at which we forget all monomials whose exponent vectors in $\N^n$ do \emph{not} belong to $\gamma_i$ (see Notation~\ref{not:restriction}). 

The above description is standard (see Proof of Item~\ref{it:ex-sol} of Proposition~\ref{prop:solutions-infty1}). The main part of this paper is exploiting it in order to outline an explicit correspondence (see Corollary~\ref{cor:main:1}), between the semi-algebraic components (algebraic in case where $\K = \C$) of $S_f$, and a subset of tuples of $\underline{\New} f$. This relation offers a straight-forward way to computing $S_f$.

\subsection{Plan of the paper} In Section~\ref{sec:clas-tor-Jel}, we aim to introduce some important notations, together with the \emph{toric non-properness} set (see Definition~\ref{def:unstab-toric-Jel}) that, similarly to the Jelonek set, keeps track of points escaping to $\K^n\setminus T_n$. 

Section~\ref{sec:faces}, has three main objectives. The first is to describe the convex geometry of the Newton polytopes (Proposition~\ref{prop:normal-fan}), which will be crucial for the rest. The second one (Section~\ref{subsec:tuple-types}) is to introduce the types of tuples of $\underline{\New} f$ that are in correspondence (mentioned above) with the components of $S_f$. The third objective is to define a toric change of coordinates (Section~\ref{subsec:mon-change}) that transforms the relevant solutions at infinity into ones in the affine hyperplanes. 

The proof, and description of this correspondence is detailed in Section~\ref{sec:main-results}. 

We apply our main results in Section~\ref{sec:compute} to express the equations of the Jelonek set from restricted polynomials. These expressions are summarized in Corollary~\ref{cor:main:1}. 

The goal of Sections~\ref{sec:app}, and~\ref{sec:very-T-BG}, is to introduce some of the applications of our method. In order to outline them here, we briefly describe the family of $T$-BG maps.

Solutions at infinity to~\eqref{eq:sys:0} which are related to $S_f$, are mapped to proper sub-tori in the boundary of a suitable compactification of $T_n$. This is defined using the \emph{algebraic moment map} (see~\cite{Ati82}) corresponding to the \emph{Minkowski sum} in $\R^n$ (see Definition~\ref{def:Minkowski}) of the members in $\underline{\New} f$. 
We require that if those sub-tori contain the common zero locus of the projective algebraic varieties defined by any subset of the set of polynomials $f_1,\ldots,f_n$, then this locus can only be the result of complete intersections. Such maps are called \emph{$T$-boundary generic}, or $T$-BG for short (see Definition~\ref{def:weakly}). 

The $T$-BG condition gives an accessible description of the Jelonek set, and at the same time, forms the complement of an algebraic variety in the space of non-proper maps. 

\subsection{Applications} %Our results have several important utilities. First, 
Corollary~\ref{cor:sufficient} describes some sufficient conditions, on the coefficients, and $\underline{\New} f$, for properness of $f$. As for the necessary ones, the following is also proven in Section~\ref{sec:app}.

\begin{corollary}\label{cor:necessary}
Let $f:\K^n\rightarrow\K^n$ be a proper non-degenerate map. Then, the union of the polytopes $\New f_1,\ldots,\New f_n\subset\R^n$, intersects all coordinate axes of $\R^n$.
\end{corollary} In the second part of Section~\ref{sec:app}, we recover $\K$-uniruledness results appearing in~\cite{Jel93}, and~\cite{Jel02} for $\K = \C$, and $\R$ respectively when $f$ is a $T$-BG generic map. 

Section~\ref{sec:very-T-BG} is devoted to studying very $T$-BG maps (see Definition~\ref{def:very}), and proving Theorem~\ref{th:very-T-BG}. The stratification property mentioned therein is made precise in Proposition~\ref{prop:strat}, and it gives a method to further simplify the computation of $S_f$ (see Remark~\ref{rem:comput}). 

\begin{remark}
We believe that descriptions in this paper can be extended, and refined, to maps $f$ that are not $T$-BG. This would be the subject of a future work. 
\end{remark}

 We note that Newton polytopes were used before as means to study properness/non-properness of maps (c.f.~\cite{Biv07,CGM96,Tha09}). However, to the best of our knowledge, the convex geometry of polytopes was not been exploited before to compute the Jelonek set.

\section{Relation to the toric non-properness set}\label{sec:clas-tor-Jel} 
We start with the following notations.
\begin{notation}\label{not:variety}

A system $h_1= \cdots = h_r = 0$ of $r$ polynomials in $l$ variables will be written as $h = \underline{0}$, with number of zeros clear from the context. Its common zero locus in $\K^l$ will be denoted by $\mathcal{Z}(h)$, and that in the $l$-torus $T_l$ will be written by $\mathcal{Z}^{\tor}(h)$. The affine space $\K^n$, spanned by points $(x_1,\ldots,x_n)$ will be identified with the set of all projective points $ [x_1:\cdots:x_n:1]$ in $\K P^n$. 
\end{notation} For a given non-degenerate map $f:\K^n\rightarrow\K^n$, one can alternatively describe the Jelonek set in the following way. A point $y\in\K^n$ belongs to $S_f$, if and only if there exists a continuous family of points $\{x(t)\}_{t\in ]0,1[}\subset T_n$ such that $|x(t)|\rightarrow +\infty$ for $t\rightarrow 0$, and $f\left( x(t)\right)\rightarrow y$. Thus, the set $\mathcal{Z}^{\tor}(f-y)$ has less isolated points than $\mathcal{Z}^{\tor}\big(f-y(t)\big)$. Using~\cite[Theorem B]{Ber75}, 

we deduce that~\eqref{eq:sys:0} has a solution $[p]=[p_0,p_1,\ldots,p_n]\in \K P^n\setminus T_n$ that is the limit of $x(t)$ when $t\rightarrow 0$. This solution $[p]\in\K P^n\setminus \K^n$ is called a \emph{strictly unstable solution} to~\eqref{eq:sys:0}. This description shows the following.

\begin{lemma}\label{lem:sol-infty}
The Jelonek set of a non-degenerate map $f:\K^n\rightarrow \K^n$ coincides with the set of points $y\in\K^n$ at which the system~\eqref{eq:sys:0} has a strictly unstable solution in $\K P^n\setminus\K^n$.
\end{lemma}

\begin{example}\label{ex:bivar-non-proper}
The non-properness set of $f:\K^2\rightarrow\K^2$, $(u,v)\mapsto\big(v(u - 1),v^2(u^2 - 3u + 2)\big)$, is $S_f=\{(t^2,t)~|~t\in\K\}\cup\{y_1 = 0\}$. 

The set of solutions $(u,v)\in\K^2$ to $f_1(u,v) - y_1 = f_2(u,v) - y_2=\underline{0}$ satisfies $v(u - 1)$ if $y=(0,0)$. Otherwise, it is equal to 
$\big( (2y_1^2 - y_2)/(y_1^2 - y_2) , (y_1^2 - y_2)/y_1\big)$,

and converges to $[0:y_2^2:0]$, or $[y_1^3:0:0]$ in $\K P^2\setminus T_2$ whenever $ y$ converges to $S_f$. 

On the other hand, the system $au + b = v +  cu^2 + d =0$ has one solution at infinity regardless of the non-zero values of the complex numbers $a,b,c,d$. This solution is thus not strictly unstable. \end{example} 
%We extend the definition of strictly unstable solutions to points in $\K^n\setminus T_n$.
Parallel to the Jelonek set, we will be studying the following.
\begin{definition}\label{def:unstab-toric-Jel}
The \emph{toric non-properness set} $S^*_f$ of a non-degenerate map $f:\K^n\rightarrow\K^n$ is a subset of $\K^n$ such that for any $y\in S^*_f$, the system~\eqref{eq:sys:0} has a strictly unstable solution in $\K^n\setminus T_n$. 
\end{definition}

In Example~\ref{ex:bivar-non-proper}, the set $S^*_f$ coincides with $\{(2t^2,t)~|~t\in\K\}\cup\{0,0\}$. 
We have the following.

\begin{lemma}\label{lem:sf-ext:non-empty}
The set $S^*_f\cup S_f\subset\K^n$ has dimension at most $n-1$.
\end{lemma}

\begin{proof}
From~\cite{Jel93,Jel02}, we have $\dim S_f\leq n-1$, and let $y$ be a point in $S^{*}_f\setminus S^*_f\cap S_f$. To compute a solution to~\eqref{eq:sys:0} in $\K^n\setminus T_n$, we set at least one of the coordinates of $x$ to be zero. This reduces the problem to computing the set of solutions in $\K^{n-1}$ to the over-determined system 
\begin{equation}\label{eq:sys1}
f_1(x_1,\ldots,x_{n-1},0) - y_1 =\cdots = f_n(x_1,\ldots,x_{n-1},0) - y_n = 0.
\end{equation} On the other hand, since $f$ is non-degenerate, the set of points $x\in \K^n$ at which the Jacobian matrix of $f$ has rank $n-1$, forms a hypersurface $J_f$ in $\K^n$. Hence, the image $f(J_f)\subset\K^n$ is an algebraic set whose dimension is at most $n-1$. Then, for any open neighborhood $\mathcal{U}_y$ of $y$, there exists a point $y'\in\mathcal{U}_y\setminus f(J_f)$ such that $f^{-1}(y')$ has a finite number of points in $\K^n$. This implies that~\eqref{eq:sys1} has finitely-many solutions in $\K^{n-1}$. Now, adding any small enough number $\varepsilon_1\in\C^*$ to the first equation, the above system will have no more solutions in $\K^{n-1}$. Thus, the point $b=y' + (\varepsilon_1,0,\ldots,0)$ satisfies $f^{-1}(b)\cap\big(\K^n\setminus T_n \big)=\emptyset$. We can iterate the above argument to show that there exists a neighborhood $\mathcal{U}_b\subset\mathcal{U}_y$ such that for any $c\in\mathcal{U}_b$, we have $f^{-1}(c)\cap\big(\K^n\setminus T_n \big)=\emptyset$. Since all of the above holds true for any such $y\in S^{*}_f\setminus S^*_f\cap S_f$, we deduce that $\K^n\setminus S^*_f\cup S_f$ is an open \emph{dense} subset of $\K^n$, and we are done.
\end{proof}

\section{Polynomials from faces}\label{sec:faces} We start by recalling the following important operation in convex geometry.

\begin{definition}\label{def:Minkowski}
The \emph{Minkowski sum} of any given sets $A,B\subset\R^n$, is formed by adding each vector in $B$ to each vector in $A$, i.e., the set $A+B =\left\lbrace a+b~|~a\in A,~b\in B\right\rbrace.$\end{definition} We assume in what follows that $f:\K^n\rightarrow\K^n$ satisfies $f(\underline{0})\in T_n$. Namely, for generic $y\in\K^n$, and $i=1,\ldots,n$, the support $\supp f_i$ coincides with $\supp(f_i - y_i)$ and contains the origin. The convention here is that $P\in \K[x_1,\ldots,x_n]$ is constant (possibly identically zero), if and only if $\supp P=(0,\ldots,0)\in\N^n$. %We thus write \]f_i(x) = \sum_{w\in\mathcal{A}_i} c_w^{(i)}x^w\quad\text{for}\quad i=1,\ldots,n.\] 
Requiring $f(\underline{0})\in T_n$ is not a restriction as it is a choice of maps up to translation. On the other hand, requiring the map to be non-degenerate implies that each polynomial in $f$ has positive degree, and that $\underline{\New} f$ is \emph{independent} in the Minkowski sense (see~\cite[\S 2]{Kho16} for details). This means that for every non-empty $J\subset\{1,\ldots,n\}$, we have \[\sum_{i\in J}\dim \New f_i\geq |J|.\] 
\noindent Introduced by Minkowski, independence of Newton polytopes was proven to be a necessary and sufficient condition for a general system of complex polynomials to have a finite number of solutions (see~\cite{Ber75}, and~\cite[\S 3]{Kho16}). 

\subsection{Minimized tuples and geometry of polytopes}\label{subsec:types-faces} If $\gamma,\delta$ are two tuples of $\underline{\New} f$ such that $\delta_i\subset\gamma_i$ for $i=1,\ldots,n$, we say that $\delta$ is a \emph{sub-tuple} of $\gamma$ (see Example~\ref{ex:minimized-tuples}). The latter inclusion property is denoted as $\delta\preceq\gamma$. The \emph{tuple dimension} $\tdi (\gamma)$ of $\gamma$ is the dimension of the Minkowski sum $\gamma_1+\cdots+\gamma_n\subset\R^n$. A tuple $\gamma$ above is said to be an \emph{$m$-tuple} if $\tdi(\gamma)=m$. The tuples that we are interested in are identified as follows. For a given vector $\alpha\in\Q^n$, and a convex polytope $A\subset\R^n$, consider the set \[A_\alpha=\left\lbrace a\in A~|~\langle a,\alpha\rangle ~\text{reaches its minimum}\right\rbrace.\] By intersecting $A$ with lines in $\R^n$ directed by $\alpha$, one can deduce that $A_\alpha=A\cap H_\alpha$, where $H_\alpha$ is a hyperplane in $\R^n$ whose normal vector is $\alpha$. This makes $A_\alpha$ a face of $A$ (recall the definition of a face). We say that $\alpha\in\Q^n$ \emph{minimizes} $A_\alpha$, and the latter is thus \emph{minimized by $\alpha$}. More generally, any vector $\alpha$ as above minimizes a unique face $\gamma_i$ in each member $\New f_i$ of $\underline{\New} f$. This determines uniquely an $m$-tuple $\gamma$ of $\underline{\New} f$, which will be \emph{minimized by} $\alpha$.
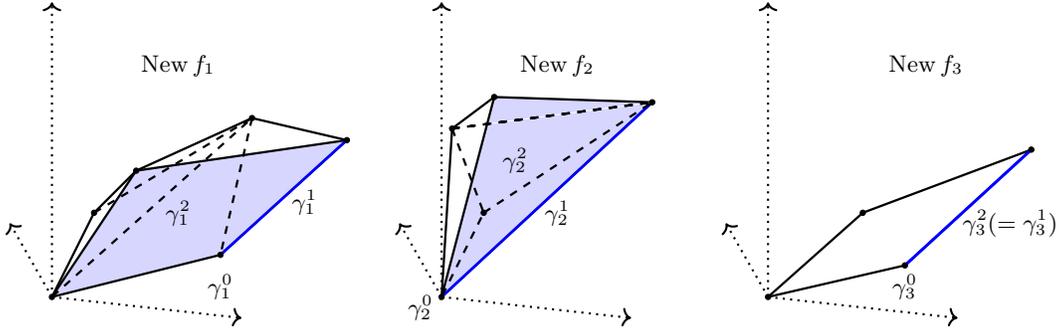
\begin{figure}
\centering
\begin{tikzpicture} [scale=1.4]
    \tikzstyle{conefill} = [fill=blue!20,fill opacity=0.8]          % style of one filling 
    \tikzstyle{ann} = [fill=white,font=\footnotesize,inner sep=1pt] % Style of annotations
    \tikzstyle{ghostfill} = [fill=white]	
    \tikzstyle{ghostdraw} = [draw=black!50]					% Style of a draw
	\tikzstyle{ann1} = [font=\footnotesize,inner sep=1pt] % Style of annotations
         
% Left Polytope

    \filldraw[conefill,line width=0.3 mm ](-5.6,0)--(-4,0.4)
    						--(-3+0.2,1.7-0.21)--(-4.8,1.2)-- cycle;	% Drawing of a cylce

	\draw[ line width=0.36 mm, color = blue] (-4,0.4)--(-3+0.2,1.7-0.21);    
    		
	\draw[ line width=0.3 mm](-5.6,0)--(-5.2,0.8)
    								--(-4.8,1.2);						% Drawing of a cylce
    						
    \draw[dashed, line width=0.3 mm](-5.2,0.8)--(-3.7,1.7)
    								--(-5.6,0);							% Drawing of a cylce 
    								
    \draw[line width=0.3 mm](-4.8,1.2)--(-3.7,1.7)--(-3+0.2,1.7-0.21);					 						
    
    \draw[dashed, line width=0.3 mm](-3.7,1.7)--(-4,0.4);   
    
% Left Axes   
    \draw[arrows=->,dotted, line width=0.8 pt](-5.6,0)--(-5.6,2.8);							% Z-axis
    \draw[arrows=->, dotted, line width=0.8 pt](-5.6,0)--(-6,0.7);						% Y-axis	
    \draw[arrows=->,dotted, line width=0.8 pt](-5.6,0)--(-3.8,-0.2);						    % X-axis       
    
% Annotations

    \node[ann1] at (-4.4,0.8) {$\gamma_1^2$};
    \node[ann1] at (-3.2,0.9) {$\gamma_1^1$};
	\node[ann1] at (-4.0,0.1) {$\gamma_1^0$}; 
	\node[ann1] at (-4.4,2.2)   {$\New f_1$};             

\begin{scope}[xshift=0.3cm]

% Middle polytope         
    \filldraw[conefill,line width=0.3 mm ](-2.2,0)--(2-2.2 ,1.7+ 0.15)    
   						--(-1.7 ,1.2+0.7 )--cycle;						% Drawing of a cylce
   						
	\draw[ line width=0.36 mm, color = blue](-2.2,0)--(2-2.2 ,1.7+ 0.15);       						
    						
	\draw[line width=0.3 mm] (0-2.2,0)--(0.1-2.2,1.6);	    						

	\draw[ line width=0.3 mm] (0.1-2.2,1.6)--(-1.7 ,1.2+0.7 );						% Drawing of a cylce						
    \filldraw[ghostfill, dashed, line width=0.3 mm ](0.1 -2.2,1.6)
    						--(2-2.2 ,1.7+ 0.15)--cycle;
    \draw[dashed, line width=0.3 mm] (0.1-2.2,1.6)--(0.4-2.2,0.8)--(0-2.2,0);	

    \draw[dashed, line width=0.3 mm] (0.4-2.2,0.8)--(2-2.2 ,1.7+ 0.15); 	
    
% Middle Axes   
    \draw[arrows=->,dotted, line width=0.8 pt](-2.2,0)--(-2.2,2.8);								% Z-axis
    \draw[arrows=->, dotted, line width=0.8 pt](0-2.2,0)--(-0.4-2.2,0.7);						% Y-axis	
    \draw[arrows=->,dotted, line width=0.8 pt](0-2.2,0)--(1.8-2.2,-0.2);						% X-axis

% Annotations

	\node[ann1] at (-1.5,1.3) {$\gamma_2^2$};  
	\node[ann1] at (-1.1,0.8) {$\gamma_2^1$};
	\node[ann1] at (-2.4,-0.1) {$\gamma_2^0$};
	\node[ann1] at (-1.1,2.2)   {$\New f_2$};             

%Dots

\draw [fill] (2-2.2 ,1.7+ 0.15) circle [radius=0.025];
\draw [fill] (0.1-2.2,1.6) circle [radius=0.025];
\draw [fill] (0.4-2.2,0.8) circle [radius=0.025];
\draw [fill] (-1.7 ,1.2+0.7 ) circle [radius=0.025];
\draw [fill] (-2.2,0) circle [radius=0.025];
\end{scope}
         
% Right polytope         

    \draw[ line width=0.3 mm] (-2.2 +3.4   ,  0  )--(-2.2 +3.4 +1.3  ,0 + 0.3  );
    \draw[ line width=0.36 mm, color = blue] (-2.2 +3.4 +1.3  ,0 + 0.3 )--(-2.2 +3.4 +1 +1.5 ,0 +1.4 );    
    \draw[ line width=0.3 mm] (-2.2 +3.4   ,  0  )--(-2.2 +3.4 +0.9  ,  0 +0.8 ) -- (-2.2 +3.4 +1 +1.5 ,0 +1.4 );

\draw[ line width=0.36 mm, color = blue] (-2.2 +3.4 +1.3  ,0 + 0.3 )--(-2.2 +3.4 +1 +1.5 ,0 +1.4 );

\begin{scope}[xshift=0cm]
% Right Axes   
    \draw[arrows=->,dotted, line width=0.8 pt](-2.2 +3.4,0)--(-2.2+3.4,2.8);							% Z-axis
    \draw[arrows=->, dotted, line width=0.8 pt](0-2.2+3.4,0)--(-0.4-2.2+3.4,0.7);						% Y-axis	
    \draw[arrows=->,dotted, line width=0.8 pt](0-2.2+3.4,0)--(1.8-2.2+3.4,-0.2);						% X-axis

\end{scope}    

	\node[ann1] at (3.5,0.7) {$\gamma_3^2(=\gamma_3^1)$};
	\node[ann1] at (2.5,0.1) {$\gamma_3^0$};  
	\node[ann1] at (2.7,2.2)   {$\New f_3$};                
	
% Dots

\draw [fill] (-5.6,0) circle [radius=0.025];
\draw [fill] (-2.2 +3.4,0) circle [radius=0.025];

\draw [fill] (-2.2 +3.4 +1 +1.5 ,0 +1.4 ) circle [radius=0.025];
\draw [fill] (-2.2 +3.4 +0.9  ,  0 +0.8 ) circle [radius=0.025];
\draw [fill] (-2.2 +3.4 +1.3  ,0 + 0.3 ) circle [radius=0.025];

\draw [fill] (-3+0.2,1.7-0.21) circle [radius=0.025];
\draw [fill] (-3.7,1.7) circle [radius=0.025];

\draw [fill] (-4.8,1.2) circle [radius=0.025]; 

\draw [fill] (-4,0.4) circle [radius=0.025]; 
\draw [fill] (-5.2,0.8) circle [radius=0.025];
 
\end{tikzpicture} 
\caption{An example of the tuple $\underline{\New} f$, together with a flag of minimized tuples for $n=3$.} \label{fig:minimized-tuples} 
\end{figure}
\begin{example}\label{ex:minimized-tuples}
In Figure~\ref{fig:minimized-tuples}, more than one direction of vectors minimize the $1$-triple $\gamma^1 = \big(\gamma^1_1,\gamma^1_2,\gamma^1_3\big)$. This is a sub-triple of the $2$-triple $\gamma^2 =\big(\gamma^2_1,\gamma^2_2,\gamma^2_3\big)$ (in transparent blue). The latter is minimized by only one vector in $\R^n$ up to scaling. 
\end{example}

A \emph{convex polyhedral cone} (we will simply say \emph{cone}) $C$ in $\R^n$ is given by a set of vectors $v_1,\ldots,v_r\in\R^n$, and defined as $\left\lbrace c_1v_1+\cdots+c_rv_r~|~c_i\in\R_{\geq 0}\right\rbrace$. Denote by $\Int C$ the relative interior of a cone $C$ above, that is $\Int C=\left\lbrace c_1v_1+\cdots+c_rv_r~|~c_i\in\R_{>0}\right\rbrace$. A \emph{face $C'$ of the cone} $C$ (or simply, face) is a cone resulting from the intersection of $C$ with a hyperplane $H_0$, passing through the origin of $\R^n$, and minimizing the value of some linear function on $C$ (see Figure~\ref{fig:cones}). An \emph{$m$-cone} is a cone of dimension $m$.

\begin{definition}\label{def:fan}
A \emph{fan} $F$ is a finite set of cones in $\R^n$ such that (see Figure~\ref{fig:2fan+poly})

\begin{itemize}

	\item if $\sigma$ is a cone in $F$, and $\tau$ is a face of $\sigma$, then $\tau$ is in $F$, and
	
	\item if $\sigma$, and $\sigma'$, are two cones in $F$, then $\sigma\cap\sigma'$ is a face of both $\sigma$, and $\sigma'$. \end{itemize}\end{definition} 

\begin{proposition}\label{prop:normal-fan} Given any $m$-tuple of $\underline{\New} f$, the set of vectors minimizing it forms an $(n-m)$-cone in $\R^n$. The union of all such cones forms a fan in $\R^n$. Moreover, an $(n-m)$-cone $C$ is a face of an $(n-d)$-cone $\tilde{C}$ if and only if the $d$-tuple $\delta$ of $\underline{\New} f$, minimized by $\tilde{C}$, is a sub-tuple of the $m$-tuple $\gamma$ minimized by $C$.
\end{proposition}

\begin{notation}
An $m$-face of a polytope $\Delta\subset\R^n$ is a face of $\Delta$ having dimension $m$. %Moreover, the above result justifies the following notation. 
The set of vectors in $\R^n$, minimizing a tuple $\gamma$ of $\underline{\New} f$ is called the \emph{minimizing cone} of $\gamma$. The fan in Proposition~\ref{prop:normal-fan} will be called the \emph{inner normal fan of $\underline{\New} f$}, and will be denoted by $\mathcal{F}(f)$. 
\end{notation}

\begin{example}\label{ex:min-cones}
Figure~\ref{fig:2fan+poly} represents an example of an inner normal fan of a couple of Newton polytopes in dimension two. For dimension three, Figure~\ref{fig:minimized-tuples} represents the $1$-triple $\gamma^1$, and the $2$-triple $\gamma^2$, respectively. They have $C^2$, and $C^1$ in $\R^3$, of Figure~\ref{fig:cones} as their respective minimizing cones. Moreover, we have $\gamma^1\preceq\gamma^2$, and $C^1$ is a face of $C^2$. 
\end{example}
\begin{proof}[Proof of Proposition~\ref{prop:normal-fan}]
To any polytope $A\in\R^n$, we associate a fan $\mathcal{F}(A)$ as follows. For each vertex $v\in \Delta$, consider the set of all \emph{facets} (i.e. $(n-1)$-faces) of $A$ adjacent to $v$. For each such a facet, we pick one minimizing vector, and collect them into a set $N_v$ in $\R^n$. The non-negative span of $N_v$ is a cone $C_v\subset\R^n$. Since $v$ is a vertex of $A$, any $m$ vectors in $N_v$ span an $m$-dimensional space in $\R^n$. Thus, the dimension of $C_v$ is equal to $n$. 

As for higher-dimensional faces, to each $k$-face $\kappa\subset\Delta$ adjacent to $v$, we associate a $(n-k)$-cone $C_\kappa\subset\R^n$ as follows. The face $\kappa$ coincides with the intersection of $n-k$ distinct facets $\xi_1,\ldots,\xi_{n-k}$ of $A$ that are also adjacent to $v$. We thus associate to $\kappa$ the cone $C_\kappa$ spanned by the elements $\alpha_1,\ldots,\alpha_{n-k}\in N_v$ that minimize $\xi_1,\ldots,\xi_{n-k}$ respectively. Hence, the resulting cone $C_\kappa$ has dimension $n-k$. 

 From the description above, we deduce that the union of all such cones, corresponding to faces $\kappa\subset A$ adjacent to $v$, constitutes the cone $C_v$. We thus define $\mathcal{F}(A)$ to be the union of all cones $C_v$, where $v$ runs through all vertices of $A$. This is a rational fan called the \emph{inner normal fan of $A$}. A more detailed description of the inner normal fan can be found in~\cite[pp. 26]{Ful93}. 

Now we are ready to prove the result for $\underline{\New} f$. Denote by $\mathcal{F}(f)$ the fan defined as the common refinement of the inner normal fans $\mathcal{F}\big(\New f_1\big),\ldots,\mathcal{F}\big(\New f_n\big)$, and let $\gamma$ be a minimized $m$-tuple of $\underline{\New} f$. We proceed by proving the three following properties.

\emph{The correspondence:} consider the intersection $C_\gamma= C_{\gamma_1}\cap\cdots\cap C_{\gamma_n}$ of cones $C_{\gamma_i}$ corresponding to $\gamma_i\subset \New f_i$ as in the paragraph above. From the definition of common refinement of fans, the set $C_\gamma$ is a cone in $\mathcal{F}(f)$. Moreover, from the description above, the cone $C_\gamma$ coincides with the set of all vectors $\alpha\in\R^n$ minimizing the tuple $\gamma$. This defines a bijective correspondence between cones in $\mathcal{F}(f)$ and minimized tuples of $\underline{\New} f$. 

\emph{The inclusion property:} Assume that there exists a minimized sub-tuple $\delta$ of $\gamma$. Then, there exists cones $C_{\delta_1},\ldots,C_{\delta_n}\subset\R^n$ such that $C_{\gamma_i}\subset C_{\delta_i}$, $\forall~i$. Moreover, from the above description, the minimizing cone $C_\delta$ coincides with the set $C_{\delta_1}\cap\cdots\cap C_{\delta_n}$, making it satisfy the inclusion $C_\gamma\subset C_\delta$. Since both latter sets are cones in $\mathcal{F}(f)$, we get $C_\gamma$ is a face of $ C_\delta$. 

As for the other direction, assume that there exists a cone $\tilde{C}$ of $\mathcal{F}(f)$, such that $C_\gamma$ is a proper face of $\tilde{C}$. Therefore, there exists a minimized tuple $\delta$ of $\underline{\New} f$ such that $C_\delta \equiv \tilde{C}$. This implies that for some $i$, we have $C_{\gamma_i}\subsetneq C_{\delta_i}$. Since both latter sets are cones in the same inner normal fan of $\New f_i$, the face $\delta_i$ is also face of the polytope $\gamma_i$. As for indexes $i\in\{1,\ldots,n\}$ such that $C_{\gamma_i}\equiv C_{\delta_i}$, using the same argument we get $\delta_i\equiv\gamma_i$. This shows that $\delta\preceq \gamma$. 

\emph{Dimension property:} Recall that a minimized face, with respect to a vector $\alpha\in\R^n$ of a polytope $A\subset\R^n$ is a subset $A_\alpha$ of $A$ in which the minimum of a linear function $\phi_\alpha: \R^n\rightarrow\R$, restricted to $A$, is reached at $A_\alpha$. This implies that $\phi_\alpha(A +B) = \phi_\alpha(A) + \phi_\alpha(B)$ for any two polytopes $A,B\subset\R^n$. Therefore, the vector $\alpha\in\R^n$ minimizes the tuple $\gamma$ if and only if $\alpha$ minimizes the set $\gamma^\oplus=\gamma_1+\ldots+\gamma_n $ of \[\New^\oplus f=\New f_1 +\cdots+ \New f_n\] as well. Note that $\gamma^\oplus$ is a face of $\New^\oplus f$. We deduce that the cone $C_\gamma$ of $\mathcal{F}(f)$, minimizing $\gamma$ coincides with the cone $C_{\gamma^\oplus}$, minimizing $\gamma^\oplus$. This makes $C_{\gamma^\oplus}$ to be a cone in the inner normal fan of $\New^\oplus f\subset\R^n$, minimizing $\gamma^\oplus$. The discussion in the beginning of this proof shows that the dimension of this cone is equal to $n-\dim \gamma^\oplus = n-\tdi\gamma$. This finishes the proof.
\end{proof} 

\subsection{Types of tuples}\label{subsec:tuple-types} Only some of the minimized tuples of $\underline{\New} f$ are relevant to our results. We start by defining one of those types. An \emph{origin face} of a polytope in $\R^n$ is a (not necessarily proper) face of that polytope having $\{\underline{0}\}$ as a vertex. If one of the members of a minimized tuple $\gamma$ of $\underline{\New} f$ is an origin face, then $\gamma$ is called a \emph{semi-origin} tuple. An \emph{origin} (resp. \emph{strictly semi-origin}) tuple is a semi-origin one whose all (resp. not all) of its members are origin faces. 
\begin{example}\label{ex:semi-origin}
Each of the triples $\gamma^0$, $\gamma^1$, and $\gamma^2$ appearing in Figure~\ref{fig:minimized-tuples} is a strictly semi-origin triple of $\underline{\New} f$. This is because $\gamma^0_2$, $\gamma^1_2$, and $\gamma^2_2$ are origin faces of $\New f_2\subset\R^3$.
\end{example} Minimized tuples $\gamma$ of $\underline{\New} f$ that are \emph{not} semi-origin, can be one of two types. Denote by $C_\gamma\subset\R^n$ the minimizing cone of $\gamma$. Choose any point $\alpha'\in\Q^n$, contained in one of the faces of $C_\gamma$ that have dimension $1$. Set $\alpha'=\alpha$ if $C_\gamma$ is already a $1$-cone. Then, Proposition~\ref{prop:normal-fan}, shows that $\alpha'$ minimizes an $(n-1)$-tuple $\gamma'$ of $\underline{\New} f$ such that $\gamma\preceq\gamma'$. For $i=1,\ldots,n$, denote by $H'_{i}\subset\R^n$ the hyperplane whose normal is $\alpha'$ and containing the $i$-th member $\gamma'_i$ of $\gamma'$, and by $H'^0_i$ the hyperplane, parallel to $H'_i$, and containing the point $\underline{0}$. 
 
The tuple $\gamma$ is said to be \emph{almost semi-origin} if there exists a pair $(\alpha',\gamma')$ as above, such that for some $i\in\{1,\ldots,n\}$, there are no integer points strictly between $H'^0_i$, and $H'_i$. 

\begin{example}\label{ex:almost-semi-origin}
The vector $(0,-1)$ minimizes the $1$-tuple $\gamma$ of horizontal edges of both triangles in Figure~\ref{fig:cones} on the right. Since there are no integer points between the origin and the shortest edge in $\gamma$, this tuple is almost semi-origin.
\end{example}

\subsection{Monomial change of coordinates}\label{subsec:mon-change} In this part, we relate the polyhedral description in Proposition~\ref{prop:normal-fan} above to some automorphisms of $T_n$ that will be crucial for our main results. The below notations are burrowed from~\cite{Ber75}. We consider the change of variables 
 \[x_j = \prod_{i=1}^nz_i^{u_{ij}},\] involving any integer matrix $U = (u_{ij})_{i,j=1,\ldots,n}$ satisfying $\det U = \pm 1$. This transformation is written as $x = z^U$, and it induces an isomorphism \[U^\star:\K[x_1,\ldots,x_n]\rightarrow \K[z_1^{\pm 1},\ldots,z_n^{\pm 1}],\] taking the monomial $x^a$ to $z^{Ua}$. Hence for any $h\in\K[x_1,\ldots,x_n]$ we have
\begin{equation}\label{eq:supp-supp}
\supp(U^\star h) = U(\supp h),
\end{equation} thus making the two sets $\mathcal{Z}^{\tor}(f - y) $, and $\mathcal{Z}^{\tor}\big(U^\star( f - y)\big)$ isomorphic. Indeed, since $\det U=\pm 1$, the map $T_n\rightarrow T_n$, $z\mapsto z^U$ is one-to-one. In particular, we have $|\mathcal{Z}^{\tor}(f - y)| = \big|\mathcal{Z}^{\tor}\big(U^\star( f - y)\big)\big|$ if this set is finite. Our focus is on a particular type of change of variables.

Consider an $n$-cone $C^{n}\subset\R^n$ as in Proposition~\ref{prop:normal-fan} minimizing a $0$-tuple $\gamma^0$ of $\underline{\New} f$. Let \begin{equation}\label{eq:cones-tower}
 C^{1}\subset \cdots\subset C^{n},
\end{equation} be any given choice of a flag of cones in $\mathcal{F}(f)$ such that for $j=1,\ldots,n-1$, the $j$-cone $C^j$ is a face of $C^{j+1}$. Proposition~\ref{prop:normal-fan} shows that this produces a flag of tuples \[\gamma^0\preceq\cdots\preceq\gamma^{n-1}\] of $\underline{\New} f$, such that $\gamma^{n-j}$ is an $(n-j)$-tuple, minimized by $C^{j}$ (see e.g. Figure~\ref{fig:cones} on the left).

\begin{example}\label{ex:flag-cones}
The $2$-cone $C^{2}$ appearing in Figure~\ref{fig:cones} minimizes the $1$-tuple $\gamma^1$ in Figure~\ref{fig:minimized-tuples}. Moreover, the $1$-cone $C^{1}$ minimizes the $2$-tuple $\gamma^2$, and the $3$-cone minimizes the $0$-tuple $\gamma^0$.
\end{example}

\begin{figure}
\centering
\begin{tikzpicture} [scale=1.5]
    \tikzstyle{conefill1} = [fill=red!20,fill opacity=0.8]          % style of one filling 
    \tikzstyle{conefill} = [pattern=vertical lines, pattern color=blue]          % style of one filling 
    \tikzstyle{ann} = [fill=white,font=\footnotesize,inner sep=1pt] % Style of annotations
    \tikzstyle{ghostfill} = [fill=white]	
    \tikzstyle{ghostdraw} = [draw=black!50]					% Style of a draw
	\tikzstyle{ann1} = [font=\footnotesize,inner sep=1pt] % Style of annotations

\begin{scope}[xshift = -4.2cm]

	\draw[ line width=0.5 mm, dashed] (0,0)--(-0.75,1.5);
	
	\draw[arrows=->, line width=0.36 mm, ] (-0.75,1.5)--(-1.2,2.4);   
	
	\draw[ line width=0.36 mm, ] (0,0)--(-2,1.4);    
	
	\draw[ line width=0.36 mm, ] (0,0)--(-0.5,1.5);    	
	
	\filldraw[conefill,line width=0.0 mm ](-1.2,2.4)--(0,0)--(-2,1.4);	% Drawing of a cylce
	
	\filldraw[conefill1,line width=0.0 mm ](-1.2,2.4)--(0,0)--(-0.5,1.5);	% Drawing of a cylce
	
	\filldraw[conefill1,line width=0.0 mm ](-2,1.4)--(0,0)--(-0.5,1.5);	% Drawing of a cylce   
    		
	\draw [fill] (0,0) circle [radius=0.025];	
	
% Annotations

	\node[ann1] at (-1.7,2)   {$C^{2}$};	
	\node[ann1] at (-1.1,2.6)   {$C^{1}$};	
	
\end{scope}	

\begin{scope}[xshift = -2 cm]

% First grid

	\draw[arrows=->,line width=0.2 mm, dotted] (0,0)-- (0,1.5); 
	
	\draw[arrows=->,line width=0.2 mm, dotted] (0,0)-- (1.5,0);

% First Newton Polytope

	% Dots

	\draw [fill, color = blue] (0,0) circle [radius=0.03];	
	
	\draw [fill, color = blue] (0,0.5) circle [radius=0.03];	
	
	\draw [fill, color = blue] (0.5,0.5) circle [radius=0.03];	
	
	% Triangle
	
	\draw[ line width=0.35 mm, color = blue] (0,0)--(0,0.5) -- (0.5, 0.5) -- cycle;
	
\end{scope}	
	
% Second grid  

	\draw[arrows=->,line width=0.2 mm, dotted] (0,0)-- (0,1.5); 
	
	\draw[arrows=->,line width=0.2 mm, dotted] (0,0)-- (1.5,0);

% Second Newton Polytope

	% Dots

	\draw [fill] (0,0) circle [radius=0.03];	
	
	\draw [fill] (0,1) circle [radius=0.03];
	
	\draw [fill] (0.5,1) circle [radius=0.03];	
	
	\draw [fill] (1,1) circle [radius=0.03];
	
	\draw [fill] (0,0.5) circle [radius=0.03];	
	
	\draw [fill] (0.5,0.5) circle [radius=0.03];		
	
	% Triangle
	
	\draw[ line width=0.35 mm] (0,0)--(0,1) -- (1, 1) -- cycle;
	
% Annotations

	\node[ann1] at (-2.2,0.25)   {$\gamma_1$};
	\node[ann1] at (-0.2,0.5)   {$\gamma_2$};
	
	\node[ann1] at (-1.7,0.7)   {$\delta_1$};
	\node[ann1] at (0.5,1.2)   {$\delta_2$};
	
	\node[ann1] at (-1.5,0.25)   {$\omega_1$};
	\node[ann1] at (0.7,0.5)   {$\omega_2$};
		
\end{tikzpicture}
\caption{On the left: The flag of cones corresponding to the flag of tuples appearing in Figure~\ref{fig:minimized-tuples}. On the right: The Newton polytopes of polynomials respective $f_1 - y_1$ and $f_2 - y_2$ from Example~\ref{ex:bivar-non-proper}}\label{fig:cones}
\end{figure}
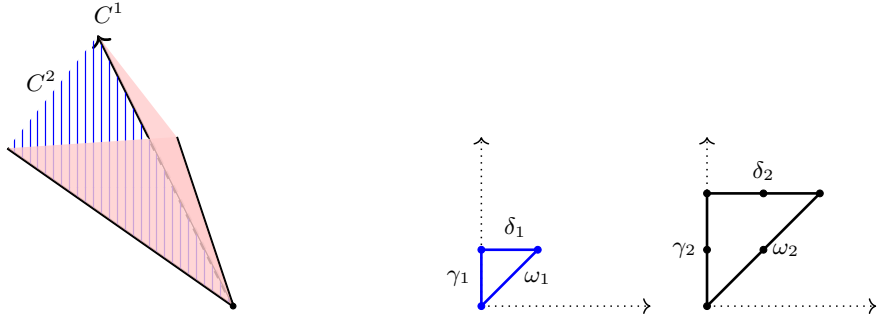

We thus obtain the inclusion $\gamma^{j-1}_1 +\cdots + \gamma^{j-1}_n\subset \gamma^j_1 +\cdots + \gamma^j_n$ for the Minkowski sums corresponding to all elements in any $j$-tuple $\gamma^j$.  

Therefore, there exists a flag of linear subspaces $L^0\subset L^1\subset \cdots\subset L^n = \R^n$ such that $L^j$ has dimension $j$, and contains $ \gamma_1^{j}+\cdots + \gamma_n^{j}$. The set $L^0$ is the point $\gamma^0_1 +\cdots + \gamma^0_n$, and will thus be regarded as a vector in $\R^n$. Now, let $\tilde{e}_1$ be a vector in $\R^n$ directing $L^1$ so that $\gamma^{1}_1 +\cdots + \gamma_n^{1} +\left\lbrace- L^0\right\rbrace$ is contained in $\R_+ \tilde{e}_1$. More generally, each $\gamma^j_1+\cdots+\gamma_n^j+\left\lbrace-L^0\right\rbrace$ lies inside a closed half-space $H_j=\left(L^{j}\setminus L^{j-1}\right)\subset L^j$. 

This allows us to construct the new basis $\tilde{e}$ by defining recursively vectors $\tilde{e}_j\in H_{j}$ as follows.

\begin{enumerate}[label=(\roman*)]

	\item \label{it:cond:rel-int} \emph{the set $\gamma^j_1+\cdots+\gamma_n^j + \left\lbrace -L^0\right\rbrace$ belongs to the cone} $\left\lbrace c_1\tilde{e}_1+\cdots+c_j\tilde{e}_j~|~c_k\in\R_{\geq 0}\right\rbrace$,

	\item \label{it:cond:integer} \emph{the vector $\tilde{e}_j$ has integer entries}, and
			
	\item \label{it:cond:volume} \emph{the vectors $\tilde{e}_1,\ldots,\tilde{e}_j$ span the lattice $L^j\cap\Z^n$ as a $\Z$-basis}.
	
\end{enumerate}

 Such basis $\tilde{e}_1,\ldots,\tilde{e}_n\in\R^n$ exists by choosing $\tilde{e}_j\in\Z^n$ so that the angle between $\tilde{e}_j$, and the space $L^{j-1}$ is large enough. Namely, we increase continuously this angle until Condition~\ref{it:cond:rel-int} is satisfied. It will remain satisfied with any such additional angle increase. We thus make it so that the parallelepiped, in $L^j$, spanned by the vectors $\tilde{e}_1,\ldots,\tilde{e}_{j}$ does not contain integer points other than its vertices. This guarantees condition~\ref{it:cond:volume}. We represent this construction as a linear map $U:\R^n\rightarrow\R^n$ taking $\tilde{e}$ to $e$, where $e=(e_1,\ldots,e_n)$ is the canonical basis of $\R^n$ represented as the identity matrix $I_n$. Condition~\ref{it:cond:volume} above implies $\det U = \pm 1$.

 Consider a minimized $m$-tuple $\gamma$ of $\underline{\New} f$. A \emph{$\gamma$-chain transformation} (denoted by $U$) of $\underline{\New} f$ is the resulting base-change integer matrix $U$ obtained using the construction above for a certain choice of a flag of cones~\eqref{eq:cones-tower}, and such that $\gamma$ is minimized by the cone $C^{n-m}$ appearing in that flag. The usefulness of $\gamma$-chain transformations $U$ will manifest themselves when being used as a change of variables $x = z^U$.
\begin{notation}\label{not:restriction}
Let $\gamma$ be a minimized tuple of $\underline{\New} f $, and let $y$ be a point in $\K^n$. For $i=1,\ldots,n$ denote by $(f_i - y_i)_\gamma$ the polynomial
\begin{equation}\label{eq:two-poss-pol}
\sum_{a\in \gamma_i\cap\supp f_i} c_a^{(i)}x^a - y_i\text{ , or }\sum_{a\in \gamma_i\cap\supp f_i} c_a^{(i)}x^a,
\end{equation} depending on whether $\underline{0}\in\gamma_i\cap\supp f_i$, or not. We thus write $(f - y)_\gamma=\big((f_1 - y_1)_\gamma,\ldots,(f_n - y_n)_\gamma \big)$. Furthermore, for any $\gamma$-chain transformation $U$ of $\underline{\New} f$, the notation $\overline{U}^\star(f-y)$ refers to the tuple consisting of element \[U^\star\big(x^{-\gamma^0_i}(f_i-y_i)\big),~i=1,\ldots,n.\] 
\end{notation} 

\begin{remark}
The use of $\overline{U}^\star$, instead of $U^\star$, for $\gamma$-chain transformations $U$, has the following advantage. For any $j=1,\ldots,n-1$, each member in the tuple $\overline{U}^\star\big( f-y\big)_{\gamma^j}$ has a constant term. 

\end{remark}

\begin{lemma}\label{lem:KhovBers2}
Given a point $y\in\K^n$, let $\gamma$ be any minimized $m$-tuple of $\underline{\New} f$. Then, there exists a $\gamma$-chain transformation $U$ of $\underline{\New} f$, involving the flag of tuples $\gamma^0\preceq \cdots\preceq\gamma^{n-1}$, such that for $i,j=1,\ldots,n-1$, we have 
\begin{enumerate}[label=(\arabic*)]
	\item \label{it:lem:restr} $ g_i=\overline{U}^\star(f_i - y_i)\in\K[z_1,\ldots,z_n]$ satisfying $\overline{U}^\star(f_i - y_i)_{\gamma^j}= g_i(z_1,\ldots,z_j,\underline{0})$, and
		
	\item \label{it:lem:const} the constant term of $g_i(z) $ depends on the point $y_i\in\K$ if and only if the $i$-th member $\gamma_i$ in the tuple $\gamma=(\gamma_1, \ldots ,\gamma_n)$ is an origin face of $\New f_i$.
\end{enumerate} 
\end{lemma} 
 
\begin{proof} \emph{Item~\ref{it:lem:restr}:} Condition~\ref{it:cond:rel-int} above shows that for $ i,j=1,\ldots,n-1$, the $i$-th member $\gamma_i^j$ in the $j$-tuple $\gamma^j$ of $\underline{\New} f$ satisfies
\[\gamma_i^j+\left\lbrace-\gamma_i^0\right\rbrace\in\left\lbrace c_1\tilde{e}_1+\cdots+c_j\tilde{e}_j~|~c_k\in\R_{\geq 0}\right\rbrace.\] Since each $z_j$ corresponds to an element $\tilde{e}_j$ in the new base-change matrix $U$, the Laurent polynomial $ x^{-\gamma_i^0}(f_i(x) -y_i)_{\gamma^j}$ is a linear combination of monomials of the form $z^{Ua}$, where all $Ua$ appearing therein belong to $\left\lbrace c_1\tilde{e}_1+\cdots+c_j\tilde{e}_j~|~c_k\in\R_{\geq 0}\right\rbrace$. Condition~\ref{it:cond:integer} shows that any above linear combination is actually a polynomial in $z^{\pm 1}$. Since a cone is a linear combination involving non-negative coefficients, the above polynomials do not have negative exponents. 

\textit{Item}~\ref{it:lem:const}: %Recall that each translated face $\gamma_i +(-\gamma_i^0)$ corresponding to the tuple $\gamma$ is also spanned by the set $(\tilde{e}_1,\ldots,\tilde{e}_{n-1})$. 
If $\gamma$ is a semi-origin $m$-tuple, then the corresponding minimizing cone $C^{n-m}$ is contained in an $n$-cone $C^n$, minimizing a semi-origin $0$-tuple $\gamma^0$ (see Proposition~\ref{prop:normal-fan}). The result follows by choosing the flag~\eqref{eq:cones-tower}, so that the $n$-cone therein is the above $C^{n}$.
\end{proof} Now, we can define more precisely $T$-BG maps. A \emph{generic} solution to a polynomial system is one at which the Jacobian matrix (evaluated using local coordinates) has full rank. Consider a $\gamma$-chain transformation $U$ as in Lemma~\ref{lem:KhovBers2} for a given minimized $m$-tuple $\gamma$ of $\underline{\New} f$. We say that $f$ is \emph{$\gamma$-generic} if for any $I\subset\{1,\ldots,n\}$, all solutions $z$ in $T_m\times\{0\}^{n-m}$ to the transformed system $\big(\overline{U}^\star f_i\big) =0,~i\in I$, are generic.

\begin{definition}\label{def:weakly}
A map $f:\K^n\rightarrow\K^n$ is said to be \emph{$T$-boundary generic} if it is a non-degenerate polynomial map, and $\gamma$-generic for every minimized $m$-tuple $\gamma$ of $\underline{\New} f$ that satisfies $m\geq 1$.
\end{definition}

\section{The two main results}\label{sec:main-results} In what follows, we assume that $f$ is a $T$-BG map. 
\begin{definition}\label{def:basic-tuple}
A tuple $\gamma$ of $\underline{\New} f$ is referred to as \emph{basic} if all its members belong to the same coordinate hyperplane of $\R^n$.
\end{definition} %We use the convention that the $0$-tuple $(\underline{0},\ldots,\underline{0})$ of $\New f$ is both basic and non-basic.

\begin{proposition}\label{prop:solutions-infty1} 
Let $y$ be a point in $ S^*_f\cup S_f$. Then, either $y=f(\underline{0})$, or there exists a minimized $m$-tuple $\gamma$ of $\underline{\New} f$, with $m\geq 1$, and a $\gamma$-chain transformation $U$ as in Lemma~\ref{lem:KhovBers2} such that

\begin{enumerate}[label=\textbf{(\alph*)}]

	\item \label{it:ex-sol} $\overline{U}^\star(f-y)=\underline{0}$ has a solution in $T_{m}\times\{0\}^{n-m}$,  
	
	\item \label{it:basic} if $y$ belongs to $S^*_f$ (resp. $S_f$), then $\gamma$ is (resp. is not) basic, 

	\item \label{it:pos-dim} if $\dim\gamma_i=0$ for some $i\in\{1,\ldots,n\}$, then $\gamma_i =(0,\ldots,0)\in\R^n$, and

	\item \label{it:semi-or} the tuple $\gamma$ is either semi-origin, or it is an almost semi-origin tuple but in this case the system $\overline{U}^\star(f-y)=\underline{0}$ has a solution in $T_{m}\times\{0\}^{n-m}$ that is \emph{not} generic. 
\end{enumerate}

\end{proposition} 

\begin{proof} In what follows, we suppose that $y\in S^*_f\cup S_f\setminus f(\underline{0})$. Then, the system~\eqref{eq:sys:0} has a strictly unstable solution $[p]\in\K P^n\setminus T_n$. The proof of Item~\ref{it:ex-sol} relies on a classical method appearing, for example, in~\cite{Ber75,Kho97,Biv07,Tha09}. 

\emph{Item~\ref{it:ex-sol}}: From Lemma~\ref{lem:sf-ext:non-empty}, we have the following. There exists a point $\mu=(\mu_1,\ldots,\mu_n)\in\K^n $ such that $\{y + t\cdot\mu \}_{t\in[0,1[}\cap \big( S^*_f\cup S_f\big) = y$, and $f - y -t\cdot\mu =\underline{0}$ has a generic solution $x(t) \in T_n$ converging to $[p]$ whenever $t\rightarrow 0$. Using Puiseux Theorem on $f - y -t\cdot\mu =\underline{0}$ regarded now as a system having Puiseux series for coefficients, we deduce that for $i=1,\ldots,n$, the coordinate $x_i(t)$ is expressed as a non-zero univariate Laurent polynomial/series of the form 
\begin{equation}\label{eq:Puiseux3}
s_it^{\alpha_i} + ~\text{higher order terms}.
\end{equation} Now plugging $x(t)$ into $f - y -t\cdot\mu =\underline{0}$, for $i=1,\ldots,n$ we deduce that there exists a tuple $\gamma$ of $\underline{\New} f$, minimized by $\alpha$, such that 
\begin{equation}\label{eq:min-limit}
t^{-\min(\sum \alpha_j a_j~|~a\in\supp f_i)}\big(f_i(x(t)) - y_i - t\mu_i\big)\rightarrow  (f_i(s) - y_i)_\gamma\quad \text{for}\quad t\rightarrow 0.
\end{equation} Hence, the point $s\in T_n$ is a solution to $(f - y)_\gamma =\underline{0}$. This means that $(\rho_1,\ldots,\rho_n)\in T_n$, satisfying $s=(\rho_1,\ldots,\rho_n)^U$, belongs to $\mathcal{Z}^{\tor}\big(\overline{U}^\star(f-y)_\gamma\big)$, where $U$ is a well-chosen $\gamma$-chain transformation. On the other hand, Item~\ref{it:lem:restr} shows that $\overline{U}^\star(f-y)_\gamma$ is a tuple of $n$ polynomials in $m$ variables. Therefore, the point $(\rho_1,\ldots,\rho_{m})\in T_{m}$ is a solution to $\overline{U}^\star(f-y)_\gamma$. We use Item~\ref{it:lem:restr} again to show that $\overline{\rho}=(\rho_1,\ldots,\rho_m,0,\ldots,0)\in\K^n$ is a solution to $\overline{U}^\star(f-y)=\underline{0}$.

\textit{Item~\ref{it:basic}:} Note that the convergence of $x(t)$, is directed towards $[p]$ with a coordinate-wise factor of $\alpha\in\Q^n$. From the expression in~\eqref{eq:Puiseux3}, we deduce the following. If $y$ belongs to $S^*_f$ (resp. $S_f$), then the point $[p]$ above belongs to $\K^n\setminus T_n$ (resp. $\K P^n\setminus \K^n$), and thus none (resp. at least one) of the coordinates of $\alpha\in\Q^n$ is/are negative. Since the members in $\underline{\New} f$ belong to $(\R_{\geq 0})^n$, the claim will thus follow by observing that the tuple $\gamma$, minimized by $\alpha$, is basic if and only if $\alpha$ has non-negative coordinates.

\textit{Item~\ref{it:pos-dim}:} Assume that for some $i\in\{1,\ldots,n\}$, the member $\gamma_i$ in the tuple $\gamma$ has dimension zero. Then $(f_i - y_i)_{\gamma}$ consists of only one monomial (recall Notation~\ref{not:restriction}). Since none of the coordinates of $s$ are zero (see proof of Item \ref{it:ex-sol}), the above monomial is not a variable. We thus obtain $(f_i - y_i)_\gamma = f_i(\underline{0}) - y_i$, and the equality $f_i(\underline{0}) - y_i=0$ is independent of the variables $x$ in $\K^n$. One can use the same above arguments to show that either $m\geq 1$, or $y=f(\underline{0})$.

\textit{Item~\ref{it:semi-or}:}  Assume that $\gamma$ is not a semi-origin tuple. The first step is to prove that $\overline{U}^\star(f-y)=\underline{0}$ has a solution $\overline{\zeta}\in T_{m}\times\{0\}^{n-m}$ with positive multiplicity. This will be done by showing that $\overline{\zeta}$ is a solution to $\overline{U}^\star(f-y - t\cdot\mu )=\underline{0}$ for any $t\in[0,1[$ that has positive multiplicity for $t=0$. 

Recall that for $t\in ]0,1[$ the solution $x(t)\in T_n$, from Item~\ref{it:ex-sol}, converges to $[p]$ with direction that depends on a vector $\alpha\in\Q^n$ minimizing $\gamma$. The point $z(t)\in T_n$, satisfying $x(t) = z^U(t)$ is also a Puiseux series
\begin{equation}\label{eq:Puiseux2}
r_it^{\lambda_i} + ~ \text{higher order terms, }i=1,\ldots,n,
\end{equation} converging to some point in $\K P^n$, with direction that depends on the vector $\lambda\in \Q^n$, and belongs to $\mathcal{Z}^{\tor}\big(\overline{U}^\star(f-y - t\cdot\mu)\big)$. We start by showing that $\lim_{t\rightarrow 0}z(t)$ belongs to the same linear subspace $T_m\times\{0\}^{n-m}$ of $\K^n$, that contains $\overline{\rho}$. As in the proof of Item~\ref{it:ex-sol}, we deduce that the vector $\lambda$ minimizes a tuple $\delta=(\delta_1,\ldots,\delta_n)$ of $\underline{\New}\big(\overline{U}^\star(f - y -\mu t)\big)$. Hence, from $(\alpha_1,\ldots,\alpha_n) = U^{T}\cdot\big(\lambda_1,\ldots,\lambda_n\big)$ (here, $T$ denotes the transpose of a matrix), we have $U\gamma = \delta$. Thus, from~\eqref{eq:supp-supp}, we get \[\overline{U}^\star(f - y -t\cdot\mu)_\gamma = \big(\overline{U}^\star(f - y -t\cdot\mu)\big)_{U\gamma} = \big(\overline{U}^\star(f - y -t\cdot\mu)\big)_\delta .\] Now, Item~\ref{it:lem:restr} shows that each of the members $\delta_i$ above lie in the subspace $\R^m\times\{0\}^{n-m}$ of $\R^m$. Hence, we have $\lambda=(0,\ldots,0,\lambda_{m+1},\ldots,\lambda_n)\in\Q^n$, and thus $\overline{U}^\star(f-y)=\underline{0}$ has a solution \[\overline{\zeta}=\lim_{t\rightarrow 0}z(t) = (\zeta_1,\ldots,\zeta_m,0,\ldots,0)\in  T_m\times\{0\}^{n-m}.\] 

Therefore, there is a solution $z(t)\in T_n$ to $\overline{U}^\star(f - y - t\cdot \mu)=\underline{0}$ for any $t\in [0,1[$, which converges to $\overline{\zeta}\in T_m\times\{0\}^{n-m}$ for $t\rightarrow 0$. On the other hand, Item~\ref{it:lem:restr} again shows that $\overline{\zeta}$ is a solution to $\overline{U}^\star(f-y)_\gamma=\underline{0}$, and from Item~\ref{it:lem:const}, we have $\overline{U}^\star(f-y)_\gamma = \overline{U}^\star f_\gamma$. This implies that $\overline{\zeta}$ is a solution to $\overline{U}^\star(f - y -t\cdot \mu)=\underline{0}$ for any $t\in [0,1[$. Since $\overline{U}^\star(f - y -t\cdot \mu)\rightarrow\overline{U}^\star(f - y)$, the point $\overline{\zeta}$ thus becomes a solution to $\overline{U}^\star(f - y)$ with positive multiplicity. By abuse of notation, we write $\overline{\rho}=\overline{\zeta}$. 

We finish by showing that $\gamma$ is an almost semi-origin tuple. From the discussion in the above paragraph, and since $f$ is $T$-BG, the point $\overline{\rho}$ is a simple solution to $\overline{U}^\star f=\underline{0}$. Therefore, the equation of the Jacobian $\Jac_{\overline{\rho}}(g)$ of $g=\overline{U}^\star(f - y)$, evaluated at $\overline{\rho}$, must depend on the point $y\in\K^n$. %since otherwise the value of $t$ would not effect multiplicity of $\hat{\rho}$. 
We deduce from Item \ref{it:lem:const} that, since $\gamma$ is not a semi-origin tuple, the tuple of polynomials $g(z_1,\ldots,z_{n-1},0)$ does \emph{not} depend on $y$. Therefore, for any $j\neq n$, the vector $\left.\left( \partial_j g_1,\ldots,\partial_j g_n\right) \right|_{z_n = 0} $ does not depend on $y$. On the other hand, if for some $i\in\{1,\ldots,n\}$, the element $\left.\partial_n g_i\right|_{z_n = 0}$ depends on $y_i$, then the polynomial $g_i(z)$ contains a term $(f_i(\underline{0}) - y_i)z^w$, where $w=(w_1,\ldots,w_{n-1},1)$. Indeed, otherwise we have $\partial_n(f_i(\underline{0}) - y_i)z^w|_{z_n=0} = 0$. 

Thus, there exists a minimized $(n-1)$-tuple $\delta'=(\delta'_1,\ldots,\delta'_n)$ of $\New g$, such that the above $\delta$ is a sub-tuple of $\delta'$, and each member $\delta_i'$ in $\delta'$ belongs to $H_n=\{X\in\R^n~|~X_n = 0\}$. 
%Clearly, such a tuple exists from Item~\ref{it:lem:restr}, and a suitable $\gamma$-chain transformation. 
Then, the point $w\in\N^{n-1}\times\{1\}$ above belongs to $H_n + e_n$. Since the coefficient in front of $z^w$ depends on $y_i$, we obtain $\underline{0} = Uw$. This makes $\gamma'$, satisfying $\delta' = U\gamma'$, an almost-origin $(n-1)$-tuple of $\underline{\New} f$. Finally, since $\delta\preceq \delta'$, we get $\gamma\preceq\gamma'$, and thus $\gamma$ is also an almost semi-origin tuple of $\underline{\New} f$.
\end{proof} 
For the other direction, we need the following notion. The \emph{origin-certification set} for a semi-origin tuple $\gamma=(\gamma_1,\ldots,\gamma_n)$ of $\underline{\New}f$, is a (possibly non-proper) subset $\theta$ of $\{1,\ldots,n\}$ satisfying $i\in\theta\Leftrightarrow \gamma_i$ is an origin tuple of $\New f_i$. The set $\{1,\ldots,n\}\setminus\theta$ will be denoted by $\theta^c$.

\begin{proposition}\label{prop:solutions-infty2}
Let $\gamma$ be a minimized $m$-tuple of $\underline{\New} f$, with $m\geq 1$. Assume that there exists a $\gamma$-chain transformation $U$ as in Lemma~\ref{lem:KhovBers2}, such that $\overline{U}^\star(f-y) =\underline{0}$ has a solution $\overline{\rho}\in  T_m\times\{0\}^{n-m}$ for some $y\neq f(\underline{0})$ in $\K^n$. If $\gamma$ is a semi-origin tuple, or if $\overline{\rho}$ is \emph{not} generic, then $y$ belongs to $S^*_f\cup S_f$. Moreover, if $\gamma$ is (resp. not) basic, then $y$ belongs to $S^*_f$ (resp. $S_f$).
\end{proposition}

\begin{proof}

Consider a smooth generic curve $\varphi:[0,1[\rightarrow \K^n$, $t\mapsto \varphi(t)$, such that $\varphi([0,1[)\cap\big( S^*_f\cup S_f\big) = \varphi(0) = y$, and that each of its coordinates is a Puiseux series in $t$. If this creates a continuous family $\{z(t)\}_{t\in ]0,1[}\subset T_n$ of \emph{isolated} solutions to $\overline{U}^\star(f -\varphi(t)) = \underline{0}$, such that $\lim_{t\rightarrow 0}z(t) =\overline{\rho}$, we get the following. Since $\overline{\rho}\in \K^n\setminus T_n$, the point $z^U(t)$ is an isolated solution to $f - \varphi(t) =\underline{0}$, whose limit $[p]=\lim_{t\rightarrow 0}z^U(t)$ belongs to $\K P^n\setminus T_n$. This makes $[p]$ a strictly unstable solution to $f - y=\underline{0}$, and hence $y\in S^*_f\cup S_f$. 

 In what follows, we show existence of such $\varphi$, creating isolated solutions $z(t)$ as above in three different cases. Let $\theta\subset\{1,\ldots,n\}$ be the origin-certification set for $\gamma$.

\textbf{First case:} $\theta=\{1,\ldots,n\}$. Item \ref{it:lem:const} implies that for any $z\in T_n$, we have $\overline{U}^\star(f - y)(z)  = \overline{U}^\star f (z) - y $. Hence, for any continuous family $\{z(t)\}_{t\in ]0,1[}\subset T_n$ of points converging to $\overline{\rho}$, we define $\varphi(t)=\big(\varphi_1(t),\ldots,\varphi_n(t)\big)\in\K^n$, such that $\varphi_i(t) = \overline{U}^\star f_i\big(z(t)\big) $ for $i=1,\ldots,n$. We choose $\{z(t)\}_{t\in ]0,1[}$ to be outside the locus of critical points of the map $\overline{U}^\star(f-y):\K^n\rightarrow\K^n$. 

\textbf{Second case:} $\theta\neq\emptyset$ and $\theta \neq\{1,\ldots,n\}$. Consider the tuple $f_{\theta^c} - y_{\theta^c}$, consisting of only those polynomials $f_i - y_i $ whose subscripts do not belong to $\theta$. Item \ref{it:lem:const}, shows that $\overline{U}^\star(f_{\theta^c} - y_{\theta^c})_\gamma =\overline{U}^\star f_{\theta^c,\gamma}$, and Item~\ref{it:lem:restr} shows that $\overline{\rho}$ is a solution to $\overline{U}^\star f_{\theta^c,\gamma}=\underline{0}$, and to $\overline{U}^\star f_{\theta^c}=\underline{0}$. Since $f$ is $T$-BG, one can choose a smooth family of points $\{z(t)\}_{t\in ]0,1[}\subset T_n$ such that the curve $z:]0,1[\rightarrow T_n$ contains $\overline{\rho}$ in its closure, and $\overline{U}^\star f_{\theta^c}(z(t)) =\underline{0}$. Then, there exists another curve $\tilde{z}:]0,1[\rightarrow T_n$, also containing $\overline{\rho}$ in its closure, such that \[\overline{U}^\star\big( f_{\theta^c} - \varphi_{\theta^c} (t)\big)(\tilde{z}(t)) =\underline{0},\] where $\varphi_{\theta^c}:[0,1[\rightarrow \K^{|\theta^c|}$, $t\mapsto\big(\varphi_i(t) \big)_{i\in\theta^c}$ is a generic continuous curve, satisfying $\varphi_{\theta^c}(0)=y_{\theta^c} = \left( y_i\right)_{i\in\theta^c}$. Indeed, since $f$ is $T$-BG, if the solution $\overline{\rho}\in T_m\times\{0\}^{n-m}$ to $\overline{U}^\star\big( f_{\theta^c} - y_{\theta^c}\big)=\underline{0}$ is isolated, then the coordinates $\left( y_i\right)_{i\in\theta^c}$ satisfy an algebraic condition. Generic change on $y_{\theta^c}$, in the form of a curve $\varphi_{\theta^c}$ above can negate this condition. 

 On the other hand, Item \ref{it:lem:const} shows that $\overline{U}^\star(f_i - y_i) = \overline{U}^\star f_i - y_i$ for $i\in\theta$. Hence, for any $t\in]0,1[$ one can choose $\varphi_i(t)\in\K$ as above so that \[\overline{U}^\star f_\theta\big(\tilde{z}(t)\big) - \varphi_\theta(t) = \underline{0}.\] This creates a curve $\varphi:[0,1[\rightarrow \K^n$, $t\mapsto \big( \varphi_i(t)\big)_{i=1,\ldots,n}$, such that $\tilde{z}(t)$ is an isolated solution to $\overline{U}^\star\big(f - \varphi(t)\big) =\underline{0}$ for any $t\in]0,1[$. 

\textbf{Third case:} $\theta=\emptyset$. This means that $\overline{\rho}\in\K^n$ is a solution to $\overline{U}^\star(f-y)=\underline{0}$ with positive multiplicity, and hence the determinant of the Jacobian $\Jac_{\overline{\rho}}\big(\overline{U}^\star(f-y)\big)$ evaluated at $\overline{\rho}$, is equal to zero. Recall the notation $\gamma^{n-1}\subset \underline{\New} f$ for the $(n-1)$-tuple appearing in the construction of the $\gamma$-chain transformation of Lemma~\ref{lem:KhovBers2}. Since $f$ is $T$-BG, we have $\Jac_{\overline{\rho}}\big(\overline{U}^\star(f-y)\big)$ depends on $y$, and thus making $\gamma^{n-1}$ an almost semi-origin tuple (see Proof of Item \ref{it:semi-or}).

Assume first that the multiplicity of $\overline{\rho}$ is two. Since $f$ is $T$-BG, the multiplicity of $\overline{\rho}$ depends on $y$. One can thus make any generic perturbation $y\mapsto y+\varepsilon$, on $y\in\K^n$, so that $\overline{\rho}$ splits into two distinct simple solutions, one is still $\overline{\rho}$, and another one $\sigma$ becomes somewhere in $\K^n$. Note that $\sigma$ is necessarily contained in $\R^n$ if $\K=\R$, and $\overline{\rho}\in\R^n\setminus (\R^*)^n$ since Item~\ref{it:lem:const} shows that such perturbation does not change the position of $\overline{\rho}\in\R^n$. Now, we choose $\tilde{y}$ so that not to increase the solutions in $\K^n\setminus T_n$, thus putting $\sigma$ in $T_n$. Each such point $\sigma$ becomes an isolated solution to $\overline{U}^\star\big(f - \tilde{y}\big)=\underline{0}$ for any generic perturbation. Finally, since $\sigma$ converges to $\overline{\rho}$ when reversing the perturbation $y\mapsto y - \varepsilon$, this proves the claim.

Assume now that the multiplicity of $\overline{\rho}$ is higher than two. Then $y\in\K^n$ belongs to an intersection $V_1\cap\cdots\cap V_k\subset\K^n$ of algebraic subvarieties determined by higher order derivations of polynomials $\overline{U}^\star(f-y)$. Moving $y$ in the strata of the union of those subvarieties from one stratum to a higher-dimensional one, we make the multiplicity of $\overline{\rho}$ diminish by one at each step. Similarly to before, this induces a splitting of $\overline{\rho}$ that gives a simple solution $\sigma\in T_n$. We can thus conclude this case in the same way we did in the above paragraph. 

\textbf{The last statement:} Recall that the curve $\varphi(t)$, defined in the beginning of this proof has Puiseux series in $t$ as coordinates. Then, from Puiseux/Newton theorem, the continuous family $\{z(t)\}_{t\in]0,1[}\subset T_n$ of points constructed above in each case, has also Puiseux series for coordinates as in~\eqref{eq:Puiseux2}. Recall that $z^U(t)$ is an isolated solution to $f - \varphi(t) =\underline{0}$, that converges to the unstable solution $[p]$ to $f - y =\underline{0}$. Then, going through the proof of Item~\ref{it:ex-sol}, we can show that $\alpha$ minimizes the tuple $\gamma$. Indeed, since $\alpha$ directs the point $z^U(t)$ towards $[p]$, and $\alpha = U^{\tr}\cdot\lambda$. 

 Finally, if $\gamma$ is (resp. not) basic, then (resp. not) all of the coordinates of $\alpha$ are non-negative. This implies that $z^U(t)$ converges to $\K^n\setminus T_n$ (resp. $\K P^n\setminus \K^n$), and thus $[p]$ belongs to $\K^n\setminus T$ (resp. $\K P^n\setminus \K^n$). Hence, we conclude that $y$ belongs to $S^{*}_f$ (resp. $S_f$).

\end{proof}

\section{Computing the Jelonek set}\label{sec:compute} %In this section, we show how to use the above main results to compute $S^*_f\cup S_f$, with $f:\K^n\rightarrow\K^n$ being a $T$-BG map. 
Let $y\in\K^n$ be a point different from $f(\underline{0})$, and consider a minimized $m$-tuple $\gamma$ of $\underline{\New} f$, with $m\geq 1$. Then, for any $\gamma$-chain transformation $U$ as in Lemma~\ref{lem:KhovBers2}, the two polynomial tuples $(f-y)_\gamma$, and $\overline{U}^\star(f-y)_\gamma$, define polynomial maps $\K^n\rightarrow\K^n$, and $\K^m\rightarrow\K^n$ respectively. 
\begin{lemma}\label{lem:invar}
Assume that $U$ is as in Lemma~\ref{lem:KhovBers2}, and that the set
\begin{equation}\label{eq:u-bar}
\left\lbrace \left. z\in T_m\times\{0\}^{n-m}~\right|~z\in\mathcal{Z}^{\tor}\big(\overline{U}^\star(f-y)\big)\right\rbrace
\end{equation} is non-empty. Then, the image $\overline{U}^\star(f-y)_\gamma ( T_m)$ in $\K^n$ does not depend on the choice of the $\gamma$-chain transformation $U$. Moreover, the determinant $\det \Jac_z\big(\overline{U}^\star(f-y)\big)$, evaluated at $z$ in~\eqref{eq:u-bar}, is either a constant in $\K^*$, or is a function in $y$ whose zero locus is also indepedent of the choice of $U$. 
\end{lemma}

\begin{proof}
Clearly, the $\gamma$-chain transformation in Lemma~\ref{lem:KhovBers2} is not unique. We start with the first statement. Since $\det U=\pm 1$, for any $x\in T_n$ there exists a unique $z\in T_n$ such that $x=z^U$. Thus, for a given $y\in\K^n$ and $i=1,\ldots,n$, we have 
\[\big(f(x) - y\big)_\gamma = \sum_{a\in\gamma_i}c_a^{(i)}x^a = \sum_{a\in\gamma_i}c_a^{(i)}z^{Ua},\] where we abuse the notation here by writing $c^{(i)}_{\underline{0}}$ for $c^{(i)}_{\underline{0}}- y_i$. Let $\rho\in T_m$ be the projection of $z$ by forgetting the last $n-m$ coordinates. Then, from Item \ref{it:lem:restr}, we get \[ \sum_{a\in\gamma_i}c_a^{(i)}z^{Ua} = \sum_{a\in\gamma_i}c_a^{(i)}\rho^{Ua} = \overline{U}^\star(f_i - y_i)_\gamma(\rho).\] We use similar arguments to prove that for any $\rho\in T_m$, there exists a point $x\in T_n$ such that $\big(f_i(x) - y_i\big)_\gamma =  \overline{U}^\star(f_i - y_i)_\gamma(\rho)$, which implies \[\big(f_i - y_i\big)_\gamma(T_n)=\overline{U}^\star(f-y)_\gamma ( T_m).\] The statement follows since the above holds for any such choice of $U$. 

Now, we prove the second statement. Let $I_x$ be the matrix constructed by replacing the $(i,i)$-th value of an $n\times n$-unit matrix by $x_i$ for some $x\in T_n$. Then, the $(kl)$-th element in the $n\times n$-matrix product $(UI_x)\cdot \Jac_x(f-y)$ is equal to
\begin{equation}\label{eq:matrix-element}
\sum_{i=1}^nu_{ki}\big( x_i\partial_i(f_l-y_l)\big).
\end{equation} Assuming that $x\in\mathcal{Z}^{\tor}(f-y)$, the equation~\eqref{eq:matrix-element} is equal to $\sum_{i=1}^nu_{ki}\big( \partial_i\big(x_i (f_l-y_l)\big)\big)$, and thus to
\[\sum_{i=1}^nu_{ki}\Big(\sum_{a\in\supp f_l} (a_i +1)c_a^{(l)}x^a\Big) = \sum_{a\in\supp f_l}\Big( \sum_{i=1}^nu_{ki}a_i\Big) c_a^{(l)}x^a.\] 
For such $x$ above, there exists a unique $z\in T_n$ such that $x=z^U$, and $z\in\mathcal{Z}^{\tor}(g)$, where $g$ denotes the tuple of polynomials $\overline{U}^\star(f-y)$. Thus, let $w\in\Z^n$ denote the point $Ua$, from which we define $\mathcal{W}_l\subset\Z^n$ such that $U\cdot \supp f_l = \mathcal{W}_l$ for $l=1,\ldots,n$. Since the change of variables $x=z^U$ produces $z^w=x^a$, equation~\eqref{eq:matrix-element} can now be written as 
\[\sum_{w\in\mathcal{W}_l} w_k c_a^{(l)}z^w=\sum_{w\in\mathcal{W}_l} (w_k+1) c_a^{(l)}z^w=\sum_{w\in\mathcal{W}_l}\partial_k\left( c_a^{(l)}z^wz_k\right)=\partial_k(z_kg_l) = z_k\partial_kg_l.\] This implies $ (UI_x)\Jac_x(f-y)= (I_z)\Jac_z(g),$ and thus \begin{equation}\label{eq:Jacob-relation}
\det\Jac_z(g) = \pm\det \Jac_x(f-y)\prod_{i=1}^n x_i/z_i,
\end{equation} for any $y\in\K^n$, and any couple $(x,z)\in \mathcal{Z}^{\tor}(f-y)\times\mathcal{Z}^{\tor}(g)\subset (T_n)^2$ related by $x=z^U$. 

On the other hand, we have $\Jac_z(g)$ cannot be identically zero for all $y$ since $f$ is $T$-BG. Namely, if $\overline{\rho}$ belongs to~\eqref{eq:u-bar}, then $\Jac_{\overline{\rho}}(g)$ is either a non-zero constant, or it depends on $y\in\K^n$.  

Assume that $\Jac_{\overline{\rho}}(g)$ depends on $y$, and let $\{\varphi(t)\}_{t\in]0,1[}\subset \K^n$ be a continuous generic family of points, such that $\lim_{t\rightarrow 0}\varphi(t) = y$. Then, this produces a family of points $\{z(t)\}_{t\in]0,1[}$, such that $z(t)$ is a solution to $ \mathcal{Z}^{\tor}\big(\overline{U}^\star (f-\varphi(t))\big)$, that converges to $\overline{\rho}$ in~\eqref{eq:u-bar} whenever $t\rightarrow 0$. Indeed, since at least one of~\eqref{eq:u-bar}, or $\Jac_{\overline{\rho}}(g)$ depend on $y$, and $f$ is $T$-BG.

Therefore, there exists a continuous family $\{x(t)\}_{t\in]0,1[}$, of points in $\mathcal{Z}^{\tor}(f-\varphi(t))$ such that $x(t) = z^U(t)$, and $\lim_{t\rightarrow 0} = [p]\in\K P^n\setminus T_n$. From~\eqref{eq:Jacob-relation}, we thus have \[\det \Jac_{x(t)}\big(f-\varphi(t)\big)=0\Leftrightarrow \det\Jac_{z(t)}\big(\overline{U}^\star(f-\varphi(t))\big) =0.\] The left hand side in the equation above does not depend on the choice of $U$. This extends to the limit $[p]$ whenever $z(t)$ reaches $\overline{\rho}$. This finishes the proof. 
\end{proof}

From the first statement of Lemma~\ref{lem:invar}, the following notation is allowed. Choosing once and for all any $\gamma$-chain transformation $U$ as in Lemma~\ref{lem:KhovBers2}, the tuple $\overline{U}^\star f_\gamma$ will henceforth be referred to as the map $F_{\gamma}=(F_{1,\gamma},\ldots,F_{n,\gamma}):\K^m\rightarrow\K^n$. 

\subsection{Semi-origin tuples}\label{subsec:semi-origin} Let $\theta\subset\{1,\ldots,n\}$ be the origin-certification set for $\gamma$. The tuple $ \big(\overline{U}^\star f_{i,\gamma}\big)_{i\in\theta}$ will be represented as a map $F_{\theta,\gamma}:\K^m \rightarrow\K^{|\theta|}$, $z \mapsto \big(F_{i,\gamma}(z)\big)_{i\in\theta}$, and simply by $F_{\gamma}$ whenever $\theta=\{1,\ldots,n\}$. If $\gamma$ is an origin tuple, then $F_\gamma\big( T_m \big)\subset\K^n$ is called the \emph{$\gamma$-parametrized} set of $f$. 

 Assume that $\gamma$ is a strictly semi-origin tuple, and let $\theta^c$ denote $\{1,\ldots,n\}\setminus\theta$. The algebraic set $V_{\theta^c}^{\gamma}(f)= \big\lbrace z\in T_m~|~F_{i,\gamma}(z)=0,~i\in\theta^c \big\rbrace \subset T_m$ is either empty, or satisfies
\[\dim V_{\theta^c}^{\gamma}(f) = \begin{cases} m-\theta^c  &\mbox{if } m>\theta^c\\  
0 & \mbox{otherwise. } \end{cases} \] This can be deduced from $f$ being $T$-BG, and from $V_{\theta^c}^{\gamma}(f)$ not being dependent on $y$. Now, if $\Theta:\K^n\rightarrow \K^{|\theta|}$ is the projection $(y_1,\ldots,y_n)\mapsto (y_i)_{i\in\theta}$, then the subset in $\K^n$
\begin{equation}\label{eq:gamma-lift}
\left\lbrace\left. \Theta^{-1} \big( F_{\theta,\gamma}(z) \big)\right|~z\in V_{\theta^c}^{\gamma}(f)\right\rbrace,
\end{equation} is isomorphic to $\K^{|\theta^c|}\times F_{\theta,\gamma}\big(V_{\theta^c}^{\gamma}(f)\big)$. 

Since $F_{\theta,\gamma}$ is well defined, we have $V_{\theta^c}^\gamma(f)$ is empty if and only if~\eqref{eq:gamma-lift} is empty. We call the latter the \emph{$\gamma$-lifted set of $f$}. Our first corollary concerns semi-origin tuples of $\underline{\New} f$.

\begin{corollary}\label{cor:Properness-through-Faces1} 
The union $S_f^*\cup S_f$ contains the $\gamma$-parametrized and $\gamma$-lifted sets, for all minimized semi-origin tuples $\gamma$ of $\underline{\New} f$. The same is true for the set $S^*_f$ (resp. $S_f$) alone for all minimized tuples $\gamma$ above, and additionally being (resp. not) basic.
\end{corollary}

\begin{proof}
Assume first that a point $y\in\K^n$ belongs to the $\gamma$-parametrized set of $f$ corresponding to an origin tuple $\gamma$ of $\underline{\New} f$. Then, there exists $(z,y)\in T_m\times\K^n$ such that $F_\gamma(z) = y$. From Item \ref{it:lem:const}, and the description in Section~\ref{subsec:semi-origin}, we have $F_\gamma(z) - y = \overline{U}^\star f_\gamma(z) - y =\overline{U}^\star(f - y)_\gamma(z)= \underline{0}$. %Hence, the system $\overline{U}^\star(f-y)_\gamma =\underline{0}$ has a solution in $ T_m$. 
Item~\ref{it:lem:restr} thus shows that $\overline{U}^\star(f-y)=\underline{0}$ has a solution $\overline{\rho}\in T_m\times\{0\}^{n-m}$. Then, Proposition~\ref{prop:solutions-infty2} shows that $y\in S^*_f\cup S_f$ and $y\neq f(\underline{0})$. The second claim of the corollary follows from the second claim of Proposition~\ref{prop:solutions-infty2}. 

Assume now that $y$ belongs to the $\gamma$-lifted set of $f$ corresponding to a strictly semi-origin tuple $\gamma$ of $\underline{\New} f$. Then, there exists $z\in T_m$ such that 
\begin{equation}\label{eq:F-gamma-theta}
F_{i,\gamma}(z)=y_i,~\forall~i\in\theta\quad\text{and}\quad F_{i,\gamma}(z) = 0,~\forall~i\in\theta^c,
\end{equation} where $\theta$ is the origin-certification set for $\gamma$. This implies that
\[ \overline{U}^\star f_{i,\gamma}(z)-y_i =0,~\forall~i\in\theta\quad\text{and}\quad \overline{U}^\star f_{i,\gamma}(z) = 0,~\forall~i\in\theta^c.\] Item \ref{it:lem:const} implies that $\overline{U}^\star f_{i,\gamma}-y_i = \overline{U}^\star(f_{i}-y_i)_\gamma$ for $i\in\theta$, and $\overline{U}^\star f_{i,\gamma}= \overline{U}^\star(f_{i}-y_i)_\gamma$ otherwise. We conclude the result for this case the same way we did in the above paragraph.
\end{proof}

\subsection{Almost semi-origin tuples}\label{subsec:Almost} \label{sec:other-tuples} Assume that the set~\eqref{eq:u-bar}, now written as $\overline{\mathcal{Z}^{\tor}\big(\overline{U}^\star(f-y)_\gamma\big)}$, is non-empty, and denote by $g$ the tuple $\overline{U}^\star(f-y)$. The polynomial $J^\gamma\in\K[z_1,\ldots,z_n,y_1,\ldots,y_n]$, defined as the determinant of the Jacobian $n\times n$-matrix $\left(\partial_i g_j\right)_{ij}$, realizes the set \[J\!\mathcal{H}_\gamma(f)=\left\lbrace y\in\K^n~\left|~J^\gamma(z,y)=0,~z\in\overline{\mathcal{Z}^{\tor}\big(\overline{U}^\star(f-y)_\gamma\big)} \right.\right\rbrace.\] From the second statement in Lemma~\ref{lem:invar}, this set does not depend on $U$. We will call $J\!\mathcal{H}_\gamma(f)$ the \emph{$\gamma$-Jacobian hyperplanes} of $f$ due to the following.

\begin{corollary}\label{cor:Properness-through-Faces3}
Assume that $\gamma$ is \emph{not} a semi-origin tuple of $\underline{\New} f$. Then, the set $J\!\mathcal{H}_\gamma(f)$ is a (possibly empty) collection of hyperplanes in $\K^n$ contained in the Jelonek set of $f$.
Moreover, if $J\!\mathcal{H}_\gamma(f)$ is non-empty, then $\gamma$ is not basic.
\end{corollary}

\begin{proof} Since $\gamma$ is not semi-origin, Item \ref{it:lem:const} shows that $\overline{U}^\star(f-y)_\gamma=\overline{U}^\star f_\gamma$. Hence, we have 
\begin{equation}\label{eq:zy=z}
\overline{\mathcal{Z}^{\tor}\big(\overline{U}^\star(f-y)_\gamma\big)} = \overline{\mathcal{Z}^{\tor}\big(\overline{U}^\star f_\gamma\big)}.
\end{equation} On the other hand, the tuple $f$ being $T$-BG implies that points in~\eqref{eq:zy=z} are generic solutions to $\overline{U}^\star f=\underline{0}$, whose number is thus finite. Proposition~\ref{prop:solutions-infty2} shows that if these solutions belong to $J\!\mathcal{H}_\gamma(f)$, then $y$ belongs to $S^*_f$, or $S_f$ depending whether $\gamma$ is basic or not. Then, the set $J\!\mathcal{H}_\gamma(f)$ is a finite (possibly empty) collection of manifolds in $\K^n$. These manifolds are actually hypersurfaces since $J^{\gamma}(z,y)$ is a polynomial in $y$.

Recall, from the last part in the proof of Item~\ref{it:semi-or} in Proposition~\ref{prop:solutions-infty1}, that for $i=1,\ldots,n$, the coordinate $y_i$ can appear only in the $(i,n)$-th element of $\big(\partial_i g_j\big)_{ij}$ after evaluating at points $\overline{\rho}$ in~\eqref{eq:zy=z}. Hence, for any $y\in\K^n$, the polynomial $J^\gamma(\overline{\rho},y)$ has degree at most one in $y$, making $J\!\mathcal{H}_\gamma(f)$ a union of hyperplanes in $\K^n$. This concludes the first statement. 

Since an almost semi-origin minimized tuple cannot be basic, we are done.
\end{proof} 
\subsection{The full description}For a given minimized semi-origin tuple $\delta$ of $\underline{\New} f$, the notation $\mathcal{X}_\delta(f)$ will refer to the $\delta$-parametrized set of $f$ if $\delta$ is an origin tuple, and to the $\delta$-lifted set of $f$ (i.e. one appearing in~\eqref{eq:gamma-lift}), otherwise.

\begin{corollary}\label{cor:main:1}
Let $f:\K^n\rightarrow\K^n$ be a $T$-BG map. Then, the Jelonek set $S_f$ minus $f(\underline{0})$ is equal to 
\begin{equation}\label{eq:union-Gamma}
\cup_{\delta}\mathcal{X}_\delta(f)\bigcup \cup_{\delta}J\!\mathcal{H}_\delta(f),
\end{equation} where $\delta$ runs through all non-basic semi-origin/almost semi-origin tuples of $\underline{\New} f$ satisfying Item~\ref{it:pos-dim}. Moreover, the toric non-properness set $S^*_f$ coincides with the union of all $\gamma$-parametrized sets $\mathcal{X}_\gamma(f)$ of $f$, such that $\gamma$ is a basic origin tuple of $\underline{\New} f$. 
\end{corollary} 

\begin{proof} We only treat the direction ``$\subset$'' since the converse is proven in Corollaries~\ref{cor:Properness-through-Faces1}, and~\ref{cor:Properness-through-Faces3}. For a given $y\in S_f$, let $\gamma$ be an $m$-tuple, together with its corresponding $\gamma$-chain transformation $U$, both satisfying Proposition~\ref{prop:solutions-infty1}. Thus, Items~\ref{it:ex-sol} and~\ref{it:lem:restr} show that $\overline{U}^\star(f-y)_\gamma=\underline{0}$ has a solution in $ T_m$. This also implies Item~\ref{it:pos-dim} for $\gamma$. Moreover, Item~\ref{it:basic} shows that $\gamma$ is non-basic.

 Item~\ref{it:lem:const} shows that, if $\gamma$ is an origin tuple, then for some $z\in  T_m$, we have $\overline{U}^\star(f-y)_\gamma(z)=\underline{0}\Leftrightarrow F_\gamma(z) = y$. If it is a strictly semi-origin one, we have~\eqref{eq:F-gamma-theta}, where $\theta\subset\{1,\ldots,n\}$ is the origin-certification set for $\gamma$. This shows that $y$ belongs to $\mathcal{X}_\gamma(f)$. 
 
Assume that $\gamma$ is \emph{not} a semi-origin tuple of $\underline{\New} f$. Then, Item~\ref{it:semi-or} shows that $\gamma$ is almost semi-origin, and that $\overline{U}^\star(f-y) =\underline{0}$ has a solution $\overline{\rho}\in  T_m\times\{0\}^{n-m}$ that is not generic. This implies that $\det \Jac_{\overline{\rho}}\big(\overline{U}^\star(f-y)\big) = 0$, and thus $y\in J\!\mathcal{H}_\gamma(f)$. This finishes the statement for $S_f$.

The statement for $S^*_f$ follows similar steps as above. The difference here being that a basic tuple can only be be an origin one.
\end{proof}

\begin{remark}\label{rem:independence}
Lemma~\ref{lem:invar}, shows the following. Once the minimized tuple $\gamma$ of $\underline{\New} f$ is fixed, any choice of $\gamma$-chain transformation $U$ of $\underline{\New} f$ is enough to compute the sets $\mathcal{X}_\gamma (f)$, or $J\!\mathcal{H}_\gamma(f)$. 
\end{remark}

\subsection{Examples}\label{subsec:examples} We use our results to compute the set $S^*_f\cup S_f$ for two examples.

\begin{example}\label{ex:Jel-set1}
Recall the map $f:\K^2\rightarrow\K^2$ from Example~\ref{ex:bivar-non-proper}. The three vectors $(1,0)$, $(0,-1)$, and $(-1,1)$ minimize three respective couples $\gamma$, $\delta$, and $\omega$ of edges of the corresponding Newton triangles appearing in Figure~\ref{fig:cones} on the right. These are the only choices of minimized tuples that satisfy Item \ref{it:pos-dim}. Choose the three chain transformations $U$, $V$, and $W$ defined by 
\[\begin{matrix*}[r]
  -1 & 1  \\
  1 & 0
 \end{matrix*}
\quad\text{,}\quad
 \begin{matrix*}[r]
  1 & 0 \\
  0 & -1
 \end{matrix*}
 \quad \text{and}\quad 
  \begin{matrix*}[r]
  1 & 0 \\
  -1 & 1
 \end{matrix*}
 \quad\text{,}\] respectively. We thus have that $\overline{U}^\star(f - y)_\gamma$, $\overline{V}^\star(f - y)_{\delta} $, and $\overline{W}^\star(f - y)_{\omega}$ are equal to \[ (-z_1 - y_1, 2z_1^2 - y_2)\ ,\quad  (z_1 - 1, (z_1 - 1)(z_1-2)),\quad\text{and}\quad ({z}_1 - y_1, {z}_1^2 - y_2).\] Since both $\gamma$, and $\omega$ are origin couples, we compute $\mathcal{X}_\gamma(f)$, and $\mathcal{X}_{\omega}(f)$ to recover $S^*_f\setminus\{(0,0)\}=\left\lbrace(2t^2,t)~|~t\in\K^*\right\rbrace$, and $\left\lbrace(t^2,t)~|~t\in\K^*\right\rbrace\subset S_f$ respectively. On the other hand, the couple $\delta$ is an almost semi-origin one. Therefore, the $\delta$-Jacobian hyperplane is computed as the determinant of the Jacobian matrix 
 \[\begin{matrix}
  1 & 2z_1 - 3  \\
  y_1 & -2z_2y_2\\
 \end{matrix}\] of $\overline{V}^\star(f -y)$, evaluated at $z=(1,0)$. This computation will show that $J\!\mathcal{\mathcal{H}}_{\delta}(f)=\left\lbrace y_1 = 0\right\rbrace$. We thus recover $S_f^*\cup S_f$.
\end{example}

\begin{example}\label{ex:Jel-set2}
Consider any map $f:\K^2\rightarrow\K^2$ whose couple of Newton polytopes $\underline{\New} f$ is represented in Figure~\ref{fig:2fan+poly}, and let $v$ denote the only origin $0$-couple of $\underline{\New} f$. Then, we have $\mathcal{X}_v(f) = f(\underline{0})$, and from Lemma~\ref{lem:f0}, we have $f(\underline{0})\in S_f$. On the other hand, exactly three non-basic minimized $1$-couples $\delta$, $\delta'$, and $\delta''$ of $\underline{\New} f$ satisfy Item \ref{it:pos-dim}. Corollary~\ref{cor:main:1} shows that 
\[S_f = \mathcal{X}_v(f)\cup\mathcal{X}_\delta (f)\cup J\!\mathcal{H}_{\delta'}(f) \cup\mathcal{X}_{\delta''}(f),\] which contains the $\delta'$-Jacobian line, the $\delta$-parametrized, and $\delta''$-lifted sets. One can check that the latter correspond to sets of the form $\left\lbrace y_1 - c =0\right\rbrace$, $\left\lbrace (a_0 + a_1 t +a_3 t^2,~b_0 + b_1 t)~|~t\in\K^*\right\rbrace$,  and $\left\lbrace  y_1 - d  =0 \right\rbrace$ respectively, where $a_i$, and $b_j$ are some coefficients in $f$, and $c,d$ are constants.
\end{example}

\begin{figure}
\centering
\begin{tikzpicture} [scale=1.5]
    \tikzstyle{conefill1} = [fill=red!20,fill opacity=0.8]          % style of one filling 
    \tikzstyle{conefill} = [pattern = north east lines, pattern color=gray]          % style of one filling 
    \tikzstyle{ann} = [fill=white,font=\footnotesize,inner sep=1pt] % Style of annotations
    \tikzstyle{ghostfill} = [fill=white]	
    \tikzstyle{ghostdraw} = [draw=black!50]					% Style of a draw
	\tikzstyle{ann1} = [font=\footnotesize,inner sep=1pt] % Style of annotations

\begin{scope}[xshift = -5cm, yshift = 1cm]

% Fan

	\draw[arrows=->, line width=0.36 mm, ] (0,0)--(+1.8,-0.9);   
	\draw[arrows=->, line width=0.36 mm, ] (0,0)--(+1,-1);   
	\draw[arrows=->, line width=0.36 mm,] (0,0)--(-1.2,0);
	\draw[arrows=->, line width=0.36 mm,] 				(0,0)--(+1.2,0);
	\draw[arrows=->, line width=0.36 mm, ] 				(0,0)--(-1.8,+0.9); 
	\draw[arrows=->, line width=0.36 mm, ] (0,0)--(-1,+1);   

	\draw [fill] (0,0) circle [radius=0.03];
		
% Annotations

	\node[ann1] at (-1.1,1.1)   {$\alpha$};	
	\node[ann1] at (1.1,-1.1) {$\alpha'$};	
	\node[ann1] at (2,-0.95) {$\alpha''$};	
\end{scope}

\begin{scope}[xshift = -2 cm]

% First grid

	\draw[arrows=->,line width=0.2 mm, dotted] (0,0)-- (0,1.5); 
	
	\draw[arrows=->,line width=0.2 mm, dotted] (0,0)-- (1.5,0);

% First Newton Polytope

	% Dots

	\draw [fill, color = blue] (0,0) circle [radius=0.03];	
	
	\draw [fill, color = blue] (0.5,0.5) circle [radius=0.03];	
	
	\draw [fill, color = blue] (1,1) circle [radius=0.03];	
	
	\draw [fill, color = blue] (1,1.5) circle [radius=0.03];		
	
	\draw [fill, color = blue] (0.5,1) circle [radius=0.03];	
	
	% Polytope 
	
	\draw[ line width=0.35 mm, color = blue] (0,0) -- (1, 1) -- (1, 1.5) -- (0.5, 1);
	
	\draw[ line width=0.35 mm, color = blue] (0,0) -- (0.5, 1);
	
\end{scope}	
	
% Second grid  

	\draw[arrows=->,line width=0.2 mm, dotted] (0,0)-- (0,1.5); 
	
	\draw[arrows=->,line width=0.2 mm, dotted] (0,0)-- (1.5,0);

% Second Newton Polytope

	% Dots

	\draw [fill] (0,0) circle [radius=0.03];	
	
	\draw [fill] (0.5,0.5) circle [radius=0.03];
	
	\draw [fill] (1,2) circle [radius=0.03];	
	
	\draw [fill] (0.5,1.5) circle [radius=0.03];
	
	\draw [fill] (0,0.5) circle [radius=0.03];	
	
	\draw [fill] (0,0.5) circle [radius=0.03];		
	
	% Polytope
	
	\draw[ line width=0.35 mm] (0,0)--(0.5,0.5) -- (1,2) -- (0.5,1.5) -- (0,0.5);
	
	\draw[ line width=0.35 mm] (0,0)-- (0,0.5);
	
	% Annotations
	
	\node[ann1] at (-1.3,0.4)   {$\delta_1$};
	\node[ann1] at (0.4,0.16)   {$\delta_2$};		
	\node[ann1] at (-1.3,1.4)   {$\delta'_1$};
	\node[ann1] at (0.7,1.95)   {$\delta'_2$};	
	\node[ann1] at (-1.85,0.65)   {$\delta''_1$};
	\node[ann1] at (0.15,1.1)   {$\delta''_2$};	
\end{tikzpicture}
\caption{The inner normal fan of Proposition~\ref{prop:normal-fan} corresponding to the couple of polytopes on the right.
}\label{fig:2fan+poly}
\end{figure}
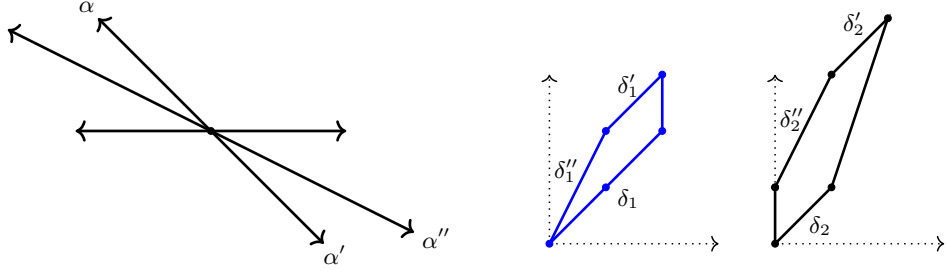

\section{(non) Properness and uniruledness}\label{sec:app} We start with the constraints for properness.

\begin{corollary}\label{cor:sufficient}
Let $f:\K^n\rightarrow\K^n$ be a $T$-BG map. Assume that $\underline{\New} f$ does not have minimized almost semi-origin tuples, and that all its minimized semi-origin tuples are basic. Then $f$ is proper. 
\end{corollary}

\begin{proof} For any given point $y\in S^*_f\cup S_f$, the absence of almost semi-origin faces implies that there are no $\gamma$-Jacobian hyperplanes. Since all the semi-origin tuples are basic, the second statement of Corollary~\ref{cor:main:1} implies that $y\notin S_f$. This makes $S_f$ to be an empty set.
\end{proof}

\begin{proof}[Proof of Corollary~\ref{cor:necessary}] For a given proper non-degenerate map $f$, let $\gamma$ be a minimized origin $m$-tuple of $\underline{\New} f$, and let $y$ be a point in the non-empty $\gamma$-parametrized set $\mathcal{X}_\gamma(f)$. It suffices to prove that \emph{any origin tuple of $\underline{\New} f$ is basic}. 

Assume that $\gamma$ is \emph{not} a basic tuple. Then, we have $m\geq 1$ since the only origin $0$-tuple is $\big(\{\underline{0}\},\ldots,\{\underline{0}\}\big)$, which is basic. Since $f$ is not necessarily $T$-BG, we cannot use Corollary~\ref{cor:main:1} directly. However, in the proof of Proposition~\ref{prop:solutions-infty2}, the statements of paragraphs labeled \textbf{First case}, and \textbf{The last statement}, apply for any non-degenerate map. Moreover, they are enough to show that $y\in S_f$.
\end{proof}

\subsection*{Proof of uniruledness}\label{subsec:unirule} First, we need the following observation.

\begin{lemma}\label{lem:f0}
For any non-degenerate map $f:\K^n\rightarrow\K^n$, whose $\underline{\New} f$ admits a non-basic minimized origin tuple, we have $f(\underline{0})\in S_f$.
\end{lemma}

\begin{proof} The existence of a non-basic tuple $\gamma$ implies that for some $i=1,\ldots,n$, say, $i=1$, none of the Newton polytopes in $\underline{\New} f$ intersects the half-ray $\R_{>0}\cdot e_1$. Indeed, otherwise the minimizing vector $\alpha\in\Q^n$, of any minimized origin tuple $\delta$ would not have negative coordinates, making $\delta$ a basic tuple. Note that this notion is closely related to the notion of \emph{convenience} of polytopes appearing in~\cite[Definition 3.1]{Biv07}, and~\cite[Definition 2.2]{Tha09}. This implies that $f(x_1,0,\ldots,0) = f(\underline{0})$ for any $x_1\in\K$. Therefore, the infinite set $\{x\in\K^n~|~x_i=0,~i\neq 1\}$ belongs now to $f^{-1}\big( f(\underline{0})\big)$. This finishes the proof.
\end{proof}

\begin{theorem}[Uniruledness for $T$-BG]\label{th:unirule}
Let $f:\K^n\rightarrow\K^n$ be a $T$-BG map. Then, for any point $y$ in $S_f$, there exists a polynomial map $\phi:\K\rightarrow S_f$ of positive degree such that $\phi(0)=y$.
\end{theorem}

\begin{proof} 
Let $y$ be a point in $S_f$. Corollary~\ref{cor:main:1} shows that $y$ belongs to one of $\mathcal{X}_\gamma(f)$, or $J\!\mathcal{H}_\gamma(f)$ for some non-basic minimized $m$-tuple $\gamma$ of $\underline{\New} f$ as in Proposition~\ref{prop:solutions-infty1}. On the other hand, both $\mathcal{X}_\gamma(f)$, and $J\!\mathcal{H}_\gamma(f)$ belong to $S_f$ (Corollaries~\ref{cor:Properness-through-Faces1}, and~\ref{cor:Properness-through-Faces3}). In what follows, we will prove the claim of the theorem for each of them. 

 Since any hyperplane in $\K^n$ is ruled by lines in $\K^n$ that are realized by linear polynomials $\phi$ above, the result is straight-forward for $\gamma$-Jacobian hyperplanes. Similarly, if $\gamma$ is a semi-origin tuple, the $\gamma$-lifted set $\mathcal{X}_\gamma(f)$ is either empty or is the fiber of a set under a non-trivial projection. Thus, it is also ruled by lines. 
 
Assume in what follows that $y\in\mathcal{X}_\gamma(f)$, with $\gamma$ now being an origin tuple. Hence, If $m=0$, the point $F_\gamma \equiv f(\underline{0})\in\K^n$ cannot be isolated in $S_f$, since $S_f$ has positive dimension (see~\cite{Jel93}, and~\cite{Jel02}). Corollary~\ref{cor:main:1} thus shows that $f(\underline{0})$, belongs to one of $\mathcal{X}_{\gamma'}(f)$, or $J\!\mathcal{H}_{\gamma'}(f)$ for some non-basic minimized $m$-tuple $\gamma$ of $\underline{\New} f$, with $m\geq 1$. We thus treat the case where $m\geq 1$.  

In the notations of Section~\ref{subsec:semi-origin}, there exists a point $\rho\in T_m$ such that $y=F_\gamma(\rho)$. The map $\phi:\K\rightarrow S_f$, $t\mapsto F_{\gamma}\big((1-t)\cdot\rho\big)$, defines a curve in $F_\gamma(\K^m)\cap S_f$ satisfying $y=\phi(0)$. Moreover, we have $\phi(0)=f(\underline{0})$ belongs to $S_f$ from Lemma~\ref{lem:f0}. This proves the claim.
\end{proof}

\section{Properties of very $T$-BG maps}\label{sec:very-T-BG} In what follows, we restrict to maps $f:\K^n\rightarrow\K^n$ having the following property.

\begin{definition}\label{def:very}
We say that $f$ is \emph{very $T$-BG} if it is $T$-BG, and for any semi-origin $m$-tuple $\gamma$ of $\underline{\New} f$, we have $\mathcal{X}_\gamma(f) =\emptyset$ whenever $m +1\leq |\theta^c| $, where $\theta$ is the origin-certifying set for $\gamma$. 
\end{definition}

\subsection{A stratification property}\label{subsec:dimension} The consequence (see Corollary~\ref{cor:strat-Jelonek}) of the following result is an accessible startification property of $S_f$ mentioned in Theorem~\ref{th:very-T-BG}.
\begin{proposition}\label{prop:strat}
Consider a minimized semi-origin $m$-tuple $\gamma$ of $\underline{\New} f$, assume that $m\leq n-2$, and let $\gamma'$ be any minimized $(m+1)$-tuple of $\underline{\New}f$ such that $\gamma\preceq\gamma'$. Then, the set $\mathcal{X}_\gamma(f)$ is either empty, or it is contained in the closure of $\mathcal{X}_{\gamma'}(f)$.
\end{proposition}

\begin{proof} First, note that $\gamma'$ is semi-origin as well. From Proposition~\ref{prop:normal-fan}, there exists a choice of a $\gamma$-chain transformation $U$ of $\underline{\New} f$, such that the two respective minimizing cones of $\gamma$, and $\gamma'$ appear consecutively in the flag~\eqref{eq:cones-tower} defining $U$. For any $y\in \mathcal{X}_\gamma(f)$, the system $\overline{U}^\star(f-y)_\gamma=\underline{0}$ has a solution $\overline{\rho}=(\rho_1,\ldots,\rho_m,\underline{0})\in T_{m}\times\{0\}^{n-m}$. In what follows, we construct a continuous family of points $\{y(t)\}_{]0,1[}\subset \mathcal{X}_{\gamma'}\subset\K^n,$ such that $y=\lim_{t\rightarrow 0}y(t)$. 
 
  Assume first that $\mathcal{X}_\gamma(f)$ is the $\gamma$-parametrized set of $f$. This makes $\gamma'$ an \emph{origin} tuple. Then, the choice of $U$ can be done so that Item~\ref{it:lem:const} is satisfied for both $\gamma$, and $\gamma'$. Let $\rho=(\rho_1,\ldots,\rho_m)\in T_m$, and $(\rho,t)$ be a point in $ T_{m+1}$, constructed by adding a non-zero parameter $t$ as an $(m+1)$-th coordinate. Thus for any $t\in\K^*$, there exists a point $y(t)\in\mathcal{X}_{\gamma'}(f)\subset\K^n$ such that $y(t)=F_{\gamma'}(\rho,t)$. Then, we combine Items \ref{it:lem:restr} and \ref{it:lem:const}, to show that $\lim_{t\rightarrow 0}y(t)=y$. 

Assume in the rest that $\gamma$ is a strictly semi-origin tuple, and that $\mathcal{X}_\gamma(f)$ is non-empty. Let $\theta\subset\{1,\ldots,n\}$ be the origin-certification set for $\gamma$. For this, we need the following statement which will be proven at the end of this proof.

\begin{claim*}
For any positive $r\ll 1$, there exists a ball $B_r(\overline{\rho})\subset\K^n$ such that 
 \begin{equation}\label{eq:ball-vspace-zero}
B_r(\overline{\rho})\cap\big( T_{m+1}\times\{0\}^{n-m-1}\big)\cap\mathcal{Z}\big( \overline{U}^\star(f_i - y_i)_{i\in\theta^c}\big)\setminus\{\overline{\rho}\}
\end{equation} is a non-empty subset of $\K^n$.
\end{claim*} The above claim shows that there exists a continuous family \[\left\lbrace\overline{\rho}(t)\right\rbrace_{t\in]0,1[} = \left\lbrace\big(\zeta(t),\underline{0}\big)\right\rbrace_{t\in]0,1[}\subset T_{m+1}\times\{0\}^{n-m-1}\] of points in the set~\eqref{eq:ball-vspace-zero} above,  %denoted by $\big(\overline{\rho}(t)\big)$, 
such that $\lim_{t\rightarrow 0}\zeta(t) = \overline{\rho}\in T_m\times\{0\}$, and
\begin{equation}\label{eq:yi-theta-c}
\big(\overline{U}^\star f_i\big)\big(\overline{\rho}(t)\big) = \big(\overline{U}^\star(f_{i}-y_i)\big)\big(\overline{\rho}(t)\big) = \underline{0},~\text{ for }~i\in\theta^c.
\end{equation} Now, one can choose the first subset of coordinates of $y(t)$. Namely, the continuous family \[\big\{\big( y_i(t)\big)_{i\in \theta}\big\}_{t\in~]0,1[}\subset  T_{|\theta|},\] of points converging to $\big( y_i\big)_{i\in \theta}$ whenever $t\rightarrow 0$, and satisfying $\overline{U}^\star f_i\big(\overline{\rho}(t)\big) - y_i(t)=0,~\text{ for }~i\in\theta$. 
Since $\gamma\preceq\gamma'$, we deduce 
\begin{equation}\label{eq:yi-prime}
\overline{U}^\star f_{i,\gamma'}\big(\zeta(t)\big) - y_i(t)=0,~i\in\theta
\end{equation} from Item~\ref{it:lem:restr}. Indeed, since $\zeta(t)$ represents the only $m+1$ non-zero coordinates of $\overline{\rho}(t)$. As for the second set of coordinates of $y(t)$, since $\gamma\preceq\gamma'$, we rewrite the equations in~\eqref{eq:yi-theta-c} as 
\begin{equation}\label{eq:yi-prime1}
\overline{U}^\star f_{i,\gamma'}\big(\zeta(t)\big) - \overline{\rho}(t)^{w_i}y_i =0,~i\in\theta^c,
\end{equation} where $w_{i}\in\N^n$. For $t\in]0,1[$, we choose the point $\big(y_i(t)\big)_{i\in \theta^c}$ as follows. If for some $i\in\theta^c$, we have $w_i\in (\N^*)^{m+1}\times\{0\}^{n-m-1}$ (i.e. $\overline{\rho}(t)^{w_i} = \zeta(t)^{w_i}$), we set $y_i(t) = \zeta(t)^{-w_{i}}\cdot \overline{U}^\star f_{i,\gamma'}\big(\zeta(t)\big)$. Otherwise, we set $y_i(t)=y_i$. As the point $\overline{\rho}(t)$ converges to $\overline{\rho}$, the family of points $y(t)$, satisfying~\eqref{eq:yi-prime}, and~\eqref{eq:yi-prime1}, converges to the initially-fixed point $y$ above. Next, we prove that this family belongs to $\mathcal{X}_{\gamma'}(f)$. 

By construction, the point $\zeta(t)\in  T_{m+1}$ is a solution to $\overline{U}^\star\big(f - y(t)\big)_{\gamma'} =\underline{0}$ for any $t\in ]0,1[$. Consider a $\gamma'$-chain transformation $U'$ as in Lemma~\ref{lem:KhovBers2}. Thus, from Lemma~\ref{lem:invar}, there exists another continuous family of points $\left\lbrace\zeta'(t)\right\rbrace_{t\in]0,1[}\subset T_{m+1}\times\{0\}^{n-m-1}$ such that \[\overline{U}'^\star\big(f-y(t)\big)_{\gamma'}\big(\zeta'(t)\big) = \underline{0}\text{, and } \overline{\rho}^U(t)=\overline{\rho}'^{U'}(t).\] From Section~\ref{subsec:semi-origin}, the above equations define a point $y(t)$ that belongs to $\mathcal{X}_{\gamma'}$. 

 We finish by proving the above claim. Since $f$ is $T$-BG, the solution $\overline{\rho}$, to $\overline{U}^\star(f_{i})=\underline{0}$ is generic. Therefore, the corresponding co-dimensions in $\K^n$, in a neighborhood of $\overline{\rho}$, satisfy \[\cdim\text{~\eqref{eq:ball-vspace-zero}}= n- \dim\left(  T_{m+1}\times\{0\}^{n-m-1}\right) + \dim \mathcal{Z}\big( \overline{U}^\star(f_i)_{i\in\theta^c} \big)  \leq n-m-1 + |\theta^c|.\] On the other hand, the set $\theta^c$ satisfies $|\theta^c|\leq m$ (since $f$ is very $T$-BG and $\mathcal{X}_\gamma(f)\neq\emptyset$). We conclude that the set~\eqref{eq:ball-vspace-zero} becomes non-empty.
\end{proof}

We omit the proof of the following consequence.

\begin{corollary}\label{cor:strat-Jelonek}
The set $\cup_{\delta}\mathcal{X}_\delta$, where $\delta$ runs through all minimized semi-origin tuples of $\underline{\New} f$ that are not basic, admits the stratification property described in Proposition~\ref{prop:strat}.
\end{corollary}

\begin{remark}\label{rem:comput}
Corollary~\ref{cor:strat-Jelonek} shows that the data of $(n-1)$-tuples is sufficient in the computation of the Jelonek set for very $T$-BG maps. 
\end{remark}

\subsection{Properties on smoothness and dimension}\label{subsec:smooth} We start with the following result.

\begin{lemma}\label{lem:smooth}
Let $\gamma$ be a minimized semi-origin $m$-tuple of $\underline{\New} f$. Then the singular locus of $\mathcal{X}_\gamma(f)$ (if non-empty) consists of complete self-intersections of its smooth subsets.
\end{lemma}

\begin{proof}
If $\gamma$ is an origin tuple, the $m\times n$-Jacobian matrix $\big(\partial_iF_{j,\gamma}\big)_{ij}$ always has full rank, since $f$ is $T$-BG. Hence, the set $\mathcal{X}_\gamma(f)$ is non-singular, and we are done. 

Assume now that $\gamma$ is a strictly semi-origin tuple, and denote by $\theta$ the origin-certification set for $\gamma$. Then, we have $\mathcal{X}_{\gamma}(f )\cong \K^{|\theta^c|}\times F_{\theta,\gamma}\left(V^{\gamma}_{\theta^c} ( f)\right)\subset\K^n$, 

where \[V^{\gamma}_{\theta^c} (f) = \left\lbrace z\in T_m~|~F_{i,\gamma}(z)=0,~i\in\theta^c\right\rbrace.\] Hence, the set $V^{\gamma}_{\theta^c} (f)$ is singular if $\mathcal{X}_{\gamma}(f)$ is singular. As before, the $ m\times |\theta^c|$-matrix $\big(\partial_iF_{j,\gamma}\big)_{ij},$ has full rank for any $z\in V^{\gamma}_{\theta^c} (f)\subset T_m$ from $f$ being $T$-BG. This finishes the proof.
\end{proof}

\begin{lemma}\label{lem:pert-non-prop}
Let $\gamma$ be a minimized semi-origin tuple of $\underline{\New} f$. Then, any perturbation $f\mapsto \tilde{f}$ in $\K$ on the non-zero coefficients appearing in the polynomials in $f$, satisfies $\dim_{\K} \mathcal{X}_\gamma (\tilde{f}) = \dim_{\K} \mathcal{X}_\gamma(f)$.
\end{lemma}

\begin{proof}
Since both cases are very similar to prove, we restrict to the case where $\gamma$ is a strictly semi-origin $m$-tuple. Denote by $\theta$ the origin-certification set for $\gamma$, and let $U$ be any $\gamma$-chain transformation as in Lemma~\ref{lem:KhovBers2}. For any $y$ in $\mathcal{X}_\gamma(f)$, there exists a point $z\in T_m$ such that~\eqref{eq:F-gamma-theta} is satisfied. Since $f$ is very $T$-BG, the set $V_{\theta^c}^\gamma(f)$, defined in Section~\ref{subsec:semi-origin}, is either empty or is the result of a complete intersection in $ T_m$ having dimension $m-|\theta^c|\geq 0$. Hence, any perturbation $f\mapsto \tilde{f}$ on the coefficients appearing in $F_{i,\gamma}(z)$, for $i\in\theta^c$, will satisfy $\dim_{\K} V_{\theta^c}^\gamma(\tilde{f})=m-|\theta^c|$. 

 On the other hand, Item~\ref{it:lem:restr} shows that $z$ and $y$ satisfying equations~\eqref{eq:F-gamma-theta} means that the point $(z,\underline{0})\in T_m\times\{0\}^{n-m} $ is a solution to $\overline{U}^\star(f-y)=\underline{0}$. Since $f$ is $T$-BG, this solution is generic independently of $y\in\K^n$. This implies that 
 \begin{equation}\label{eq:global1}
 \left\lbrace\left.(z,y)\in T_m\times\K^{|\theta|}~\right|~F_{i,\gamma}(z)=y_i,~i\in\theta\text{, and }F_{i,\gamma}(z)=0,~i\in\theta^c \right\rbrace\subset T_m\times\K^{|\theta|}
 \end{equation} is also the result of a complete intersection. Hence, for any $i\in\{1,\ldots,n\}$, the perturbation $f\mapsto \tilde{f}$, resulting from the perturbation $F_{\theta,\gamma}\mapsto\tilde{F}_{\theta,\gamma}$, on the coefficients appearing in $F_{i,\gamma}$, $i\in\theta$ will not change the dimension of~\eqref{eq:global1}. Therefore, we have \[ \dim_{\K}\tilde{F}_{\theta,\gamma}\big(V_{\theta^c}^\gamma(\tilde{f})\big)=\dim_{\K} F_{\theta,\gamma}\big(V_{\theta^c}^\gamma(f)\big).\] Finally, we multiply both above sets by $\K^{|\theta^c|}$ in order to obtain each of $\mathcal{X}_\gamma(\tilde{f})$, and $\mathcal{X}_\gamma(f)$. We deduce that $\dim_{\K}\mathcal{X}_\gamma(f)=\dim_{\K}\mathcal{X}_\gamma(\tilde{f})$.
\end{proof}

\begin{proposition}\label{prop:dim-semi-compl}
Let $\K = \C$, and let $\gamma$ be a minimized non-basic semi-origin tuple of $\underline{\New} f$. Then, there exists a minimized semi-origin tuple $\gamma'$ of $\underline{\New} f$ such that $\gamma\preceq\gamma'$, and $\dim_\C \mathcal{X}_{\gamma'}(f)=n-1$.
\end{proposition}

\begin{proof}
Let $y$ be a point in $\mathcal{X}_{\gamma}(f)$. We start by the following assumption. \emph{The point $y$ does not belong to $\mathcal{X}_\delta(f)$ (resp. $J\!\mathcal{H}_\delta (f)$) for any semi-origin (resp. almost semi-origin) tuple $\delta$ of $\underline{\New} f$, where $\gamma\npreceq\delta$}. Then, we obtain the result from the following two facts: The first uses $\dim S_f=n-1$ in~\cite{Jel93} for complex maps, from which we deduce that~\eqref{eq:union-Gamma} has dimension $n-1$ as well. The second uses Proposition~\ref{prop:strat} to show that $y$ belongs to all sets $\mathcal{X}_{\gamma'}(f)$ whose tuples $\gamma'$ satisfy $\gamma\preceq\gamma'$. Hence, at least one of those $\gamma'$ has to satisfy $\dim\mathcal{X}_{\gamma'}(f) = n-1$.

Now, we show that the above assumption can be always made without loss of generality for our proof. If there exists a semi-origin (resp. almost semi-origin) $d$-tuple $\delta$ of $\underline{\New} f$, such that $\gamma\npreceq\delta$, and $y\in\overline{\mathcal{X}_\delta(f)}$ (resp. $y\in J\!\mathcal{H}_\delta$), then we proceed as follows. 

Choose a $\gamma$-chain transformation, a $\delta$-chain transformation $U,V$ respectively, and two points $z,\zeta$ in $(\C^*)^m,(\C^*)^d$, so that $\overline{U}^\star(f - y)_{\gamma}(z)=\underline{0}$, and $\overline{V}^\star(f - y)_{\delta}(\zeta)=\underline{0}$. Since $f$ is $\C^*$-BG, for any perturbation $f\mapsto\tilde{f}$ as in Lemma~\ref{lem:pert-non-prop}, there exists a point $\tilde{y}\in\C^n$, close to $y$, and a point $\tilde{z}\in (\C^*)^m$, close to $z$, such that $\overline{U}^\star(f - \tilde{y})_{\gamma}(\tilde{z})=\underline{0}$. We choose such a perturbation by replacing $c_a$ by $c_a+\epsilon$, for some generic $\epsilon\in\C^*$, with $c_a$ is a coefficient of a polynomial $f_{i,\gamma}$, satisfying $a\in\gamma_i\setminus\delta_i$ for some $i\in\{1,\ldots,n\}$. Such coefficients exist since $\gamma$ is \emph{not} a sub-tuple of $\delta$. 

Note that $\overline{V}^\star(\tilde{f} - y)_{\delta}=\overline{V}^\star(f - y)_{\delta}$ for any $y\in\K^n$. Hence, using Lemma~\ref{lem:pert-non-prop}, we can assume that $f$ is $\C^*$-BG generic enough so that the above choice of $\epsilon$ induces a perturbation $y\mapsto \tilde{y}$, such that $\overline{V}^\star(f - \tilde{y})_{\delta}=\underline{0}$ does \emph{not} have solutions in $T_d$. Indeed, this system is over-determined with $n$ equations and less than $n$ variables, and the point $y\in \K^n$ contributes to its coefficients.

 Therefore, we can find a series of perturbations as above summed up as $f\mapsto\hat{f}$, and satisfying the assumption in the first paragraph, but for $\hat{f}$ instead of $f$. Finally, since $f$ is $\C^*$-BG, the proof follows from Lemma~\ref{lem:pert-non-prop}. 

\end{proof}

\begin{proposition}\label{prop:dim-semi-real}
Let $\K=\R$, and let $\gamma$ be a minimized non-basic semi-origin tuple of $\underline{\New} f$. Then, there exists a minimized semi-origin tuple $\gamma'$ of $\underline{\New} f$ such that $\gamma\preceq\gamma'$, and $\dim_\R \mathcal{X}_{\gamma'}(f)=n-1$.
\end{proposition}

\begin{proof}
%Assume in what follows that $y\in\mathcal{X}_\gamma(f)$. 
Let $\C f:\C^n\rightarrow\C^n$ be the map, defined as the extension of $f$ to the complex space. Note that this map is not necessarily very $\C^*$-BG. Namely, equation $f =\underline{0}$ might have complex-conjugate solutions in $ \C P^n\setminus (\C^*)^n$ that are not generic. However, Lemma~\ref{lem:pert-non-prop} shows that perturbations (in $\R$) of the non-zero coefficients of $f$ would not change the dimension of $\mathcal{X}_\gamma(f)\subset\R^n$. We thus use such real perturbations to assume, without loss of generality, that $\C f$ is very $\C^*$-BG.

For any minimized semi-origin tuple $\delta$ of $\underline{\New} f$, we have 
\[\mathcal{X}_\delta(f)\subset \R\cap\mathcal{X}_{\delta}(\C f)\subset\mathcal{X}_\delta (\C f).\] On the other hand, Proposition~\ref{prop:strat} shows that there exists a minimized semi-origin tuple $\gamma'$ of $\underline{\New} f$ such that $\gamma\preceq\gamma'$, and $\mathcal{X}_\gamma( f)\subset \overline{\mathcal{X}_{\gamma'}(f)}$. Moreover, the above inclusion holds true for $\C f$ as well. Therefore, we get the diagram 
\begin{equation}\label{eq:main-diag}
\begin{matrix}
 \mathcal{X}_\gamma(\C f) & \subset & \overline{\mathcal{X}_{\gamma'}(\C f)}  \\ 
 \cup & \ & \cup  \\
 \mathcal{X}_\gamma( f)&\subset  & \overline{\mathcal{X}_{\gamma'}( f)}  \\
\end{matrix}
\end{equation} We use Proposition~\ref{prop:dim-semi-compl} to choose $\gamma'$ above so that $\dim_\C \overline{\mathcal{X}_{\gamma'}(\C f)} = n-1$. To prove that $\dim_\R \mathcal{X}_{\gamma'} = n-1$, we proceed as follows. Recall that $\mathcal{X}_{\gamma'}(f)\subset\R^n$ is a subset of an intersection of algebraic varieties in $\R^n$ (see Section~\ref{subsec:semi-origin}). Then, inequality $\dim_\R \mathcal{X}_{\gamma'}(f) < \dim_\C \mathcal{X}_{\gamma'}(\C f)$ implies that the set $\mathcal{X}_{\gamma'}(f)$ is a subset of the singular locus of $\mathcal{X}_{\gamma'}(\C f)$. However, Lemma~\ref{lem:smooth} shows that this is not the case. This finishes the proof.
\end{proof}
\subsection{Proof of Theorem~\ref{th:very-T-BG}} Since the stratification property is already proven in Corollary~\ref{cor:strat-Jelonek}, we will only show the weak smoothness, and dimension properties. From Section~\ref{sec:compute}, the set $S_f$ is written as the union appearing in~\eqref{eq:union-Gamma}, where $\delta$ runs through all semi-origin tuples and almost semi-origin tuples of $\underline{\New} f$. Corollary~\ref{cor:Properness-through-Faces3} shows that $\cup_{\delta}J\!\mathcal{H}_\delta(f)$ is a union of hyperplanes in $\K^n$. Hence, is a union of smooth, and $(n-1)$-dimensional components in $\K^n$. 

On the other hand, Propositions~\ref{prop:strat}, and~\ref{prop:dim-semi-compl} (resp.~\ref{prop:dim-semi-real}) show that $\cup_{\delta}\mathcal{X}_\delta(f)$ is a set having dimension $n-1$ in $\K^n$ for $\K = \C$ (resp. $\R$). This shows that $\dim_\K S_f = n-1$. The weak smoothness property for the union $\cup_{\delta}\mathcal{X}_\delta(f)$ is a consequence of Lemma~\ref{lem:smooth}. \qed

\subsection*{Acknowledgement} I am grateful to Prof. Zbigniew Jelonek for introducing me to the problem, and for his helpful remarks.

\subsection*{Contact}
  Boulos El Hilany,
  Instytut Matematyczny Polskiej Akademii Nauk,
  ul. \'Sniadeckich 8
  00-656, Warsaw,
  Poland;
\href{mailto:boulos.hilani@gmail.com}{boulos.hilani@gmail.com}.

\bibliographystyle{alpha}					   % For the style

\bibliography{mainbib}

\end{document}